\documentclass[10pt]{amsart}
\usepackage{amsmath,amssymb,amsthm}
\usepackage{enumerate}             
\usepackage{srcltx}                
\usepackage{cite}                  
\usepackage{xspace}                
\usepackage{multirow}              
\usepackage{bookmark}              
\usepackage{tikz}
\usepackage{ctable}                

\newtheorem{theorem}{Theorem}[section]
\newtheorem*{theorem-nn}{Theorem}
\newtheorem{corollary}[theorem]{Corollary}
\newtheorem{lemma}[theorem]{Lemma}
\newtheorem{proposition}[theorem]{Proposition}

\newtheorem*{question-nn}{Question}

\newtheorem*{conjecture-nn}{Conjecture}
\theoremstyle{definition}
\newtheorem{definition}[theorem]{Definition}
\newtheorem*{definition-nn}{Definition}
\theoremstyle{remark}

\newtheorem*{example-nn}{Example}

\newcommand{\acts}{\curvearrowright}
\newcommand{\flow}[1][F]{\mathfrak{#1}}
\newcommand{\la}{\leftarrow}
\newcommand{\es}{\varnothing}

\DeclareMathOperator{\proj}{proj}
\DeclareMathOperator{\eff}{Eff}
\DeclareMathOperator{\sgr}{Sgr}
\DeclareMathOperator{\ssp}{Ssp}
\DeclareMathOperator{\stab}{stab}
\DeclareMathOperator{\per}{per}
\DeclareMathOperator{\spanop}{span}

\makeatletter

\def\@cber[#1]{\protect\@ifnextchar[{\@cber@i[#1]}{\@ifnextchar*{\@cberRel[#1]}{\@cber@i[#1][]}}}
\def\@cber@i[#1][#2]{\@ifnextchar[{\@cber@ii[#1][#2]}{\@cber@ii[#1][#2][]}}
\def\@cber@ii[#1][#2][#3]{\mathsf{#1}\ifx&#2&\else_{#2}\fi\ifx&#3&\else^{#3}\fi}
\def\@cberRel[#1]*{\@ifnextchar[{\@cberRel@i[#1]*}{\@cberRel@i[#1]*[]}}
\def\@cberRel@i[#1]*[#2]{\@ifnextchar[{\@cberRel@ii[#1]*[#2]}{\@cberRel@ii[#1]*[#2][]}}
\def\@cberRel@ii[#1]*[#2][#3]{\protect\kern1pt\mathsf{#1}\ifx&#2&\else_{#2}\fi\ifx&#3&\else^{#3}\fi\kern1pt}

\def\cber{\@cber[E]}
\def\fber{\@cber[F]}

\makeatother

\begin{document}
\title{On time change equivalence of Borel flows}

\author{Konstantin Slutsky}
\address{Department of Mathematics, Statistics, and Computer Science\\
University of Illinois at Chicago\\
322 Science and Engineering Offices (M/C 249)\\
851 S Morgan Street\\
Chicago, IL 60607--7045\\
United States}
\email{kslutsky@gmail.com}
\begin{abstract}
  This paper addresses the notion of time change equivalence for Borel \( \mathbb{R}^{d} \)-flows.
  We show that all free \( \mathbb{R}^{d} \)-flows are time change equivalent up to a compressible
  set.  An appropriate version of this result for non-free flows is also given.
\end{abstract}

\maketitle

\section{Introduction}
\label{sec:introduction}

A Borel flow \( \flow \) is a Borel measurable action of \( \mathbb{R}^{d} \) on a standard Borel
space.  The action will be denoted additively: \( x + \vec{r} \) is the action of
\( \vec{r} \in \mathbb{R}^{d} \) upon the point \( x \in X \).  With any flow \( \flow \) on \( X \)
we associate an equivalence relation \( \cber[X][\flow] \) given by \( x \cber*[X][\flow] y \)
whenever there is \( \vec{r} \in \mathbb{R}^{d} \) such that \( x + \vec{r} = y \). An equivalence
class of \( x \in X \) will be denoted by \( [ x ]_{\cber[X][\flow]} \). An \emph{orbit equivalence}
between two flows \( \mathbb{R}^{d} \acts X \) and \( \mathbb{R}^{d} \acts Y \) is a Borel bijection
\( \phi : X \to Y \) such that for all \( x, y \in X\)
\[ x \cber*[X] y \iff \phi(x) \cber*[Y] \phi(y).  \]

The notion of orbit equivalence is particularly suited for actions of discrete groups, but it tends
to trivialize for certain locally compact groups.  For instance, it is known that all
non smooth free \( \mathbb{R}^{d} \)-flows are orbit equivalent.  To remedy this collapse, one
often considers various strengthenings of orbit equivalence, usually by imposing ``local''
restrictions on the orbit equivalence maps.

Given any orbit of a \emph{free} action of \( \mathbb{R}^{d} \), there is
a bijective correspondence between points of the orbit and elements of \( \mathbb{R}^{d} \).  More
precisely, with an equivalence relation \( \cber[X] \) arising from a free flow
\( \mathbb{R}^{d} \acts X \) one may associate a \emph{cocycle} map
\( \rho : \cber[X] \to \mathbb{R}^{d} \) which is defined by the condition
\[ x + \rho(x,y) = y \quad \textrm{for all } x \cber*[X] y. \]
The map \( \rho(x, \, \cdot \,) \) establishes a bijection between the orbit of \( x \) and
\( \mathbb{R}^{d} \).  While concrete identification depends on the choice of \( x \), any
translation invariant structure on \( \mathbb{R}^{d} \) can be transferred onto orbits of such an
action unambiguously.  Time change equivalence between free flows is defined as an orbit equivalence
that preserves the topology on every orbit.

\begin{definition}
  Let \( \mathbb{R}^{d} \acts X \) and \( \mathbb{R}^{d} \acts Y \) be free Borel flows.  An orbit
  equivalence \( \phi : X \to Y \) is said to be a \emph{time change equivalence} if for any
  \( x \in X \) the map \( \xi(x, \,\cdot\,) : \mathbb{R}^{d} \to \mathbb{R}^{d} \) specified by
  \[ \phi(x + \vec{r}) = \phi(x) + \xi(x, \vec{r}) \]
  is a homeomorphism\footnote{There seems to be little difference whether \( \xi(x, \, \cdot \,) \)
    is also required to preserve the smooth structure.  As a matter of fact, it is usually easier to
    construct a time change equivalence which is moreover a \( C^{\infty} \)-diffeomorphism on
    every orbit.}\kern-4pt.
\end{definition}

The concept of time change equivalence has been studied quite extensively in ergodic theory, where
the set up differs from the one of Borel dynamics in the following aspects.  In ergodic theory phase
spaces are assumed to be endowed with probability measures, which flows are required to preserve (or
to ``quasi-preserve'', i.e., to preserve the null sets).  Moreover, all conditions of interest may
hold only up to a set of measure zero as opposed to holding everywhere.  In these regards ergodic
theory is less restrictive than Borel dynamics.  On the other side, all orbit equivalence maps are
additionally required to preserve measures between phases spaces, which significantly restricts the
pool of possible orbit equivalences.  In this aspect ergodic theory provides finer notions to
differentiate flows.  To summarize, frameworks of Borel dynamics and ergodic theory are in general
positions, and while methods used in these areas are intricately related, there are oftentimes no
direct implications between results.

In the measurable case, there is a substantial difference between one dimensional and higher
dimensional \( d \ge 2 \) flows.  There are continuumly many time change inequivalent \(
\mathbb{R} \)-flows (see \cite{MR653094}).  In higher dimension the situation is simpler.  Two
relevant results here are due to D.~Rudolph \cite{MR536948} and J.~Feldman \cite{MR1113569}.

\begin{theorem}[D.~Rudolph]
  \label{thm:rudolph-thm}
  Any two measure preserving ergodic\footnote{Recall that a measure is ergodic, if any invariant set
  is either null or has full measure.} \( \mathbb{R}^{d} \)-flows, \( d \ge 2 \), are time
change equivalent.
\end{theorem}

\begin{theorem}[J.~Feldman]
  \label{thm:feldman-thm}
  Any two quasi measure preserving ergodic \( \mathbb{R}^{d} \)-flows, \( d \ge 2 \), are time
  change equivalent.
\end{theorem}

A striking difference of Borel framework was discovered in \cite{MR2578608} by B.~Miller and
C.~Rosendal, where they proved that all non smooth Borel \( \mathbb{R} \)-flows are time change
equivalent.  In other words, the flexibility of considering orbit equivalences which do not preserve
any given measure turns out to be more important than the necessity to define equivalences on every
orbit (as opposed to almost everywhere).  We recall that an equivalence relation \( \cber \) on a
standard Borel space \( X \) is said to be smooth if there is a Borel map \( f : X \to
\mathbb{N}^{\mathbb{N}} \) such that 
\[ x \cber* y \iff f(x) = f(y). \]
For equivalence relations arising as orbit equivalence relations of Polish group actions this is
equivalent\footnote{We refer the reader to \cite{MR1321597} for all the relevant results from
  descriptive set theory.} to existence of a Borel transversal --- a Borel set that intersects each
equivalence class in a single point.

Miller and Rosendal posed a question whether any two free
Borel \( \mathbb{R}^{d} \)-flows are time change equivalent.  This paper makes a contribution in
this direction.





\subsection{Main results}
\label{sec:main-results}

Many constructions in Borel dynamics and ergodic theory have to deal with two kinds of issues ---
``local'' and ``global''.  Global aspects refer to properties that hold relative to many (usually
all) orbits.  Local aspects of a construction, on the other hand, reflect behavior that is local to
any given orbit.  For instance, the property of \( \phi : X \to Y \) being an orbit equivalence
between flows \( \mathbb{R}^{d} \acts X \) and \( \mathbb{R}^{d} \acts Y \) is global, as it
requires different orbits to be mapped to different orbits.  To be more specific, if a partial
construction of \( \phi \) yields \( X \ni x \mapsto \phi(x) \in Y \) for some \( x \in X \), then
no \( x' \in X \) such that \( \neg \bigl( x' \cber*[X] x \bigr) \) can be mapped into
\( \bigl[ \phi(x) \bigr]_{\cber[Y]} \).  In this sense, before defining \( \phi(x') \), any \( x' \)
needs to know something about points \( x \in X \) from other orbits.  On the other hand, the
property for an orbit equivalence \( \phi : X \to Y \) to be a time change equivalence is
purely local as it can be checked by looking at each orbit individually.

It is oftentimes beneficial (if nothing else for pedagogical reasons) to decouple whenever possible
local and global aspects of a problem at hand.  For various constructions of orbit equivalence this
is achieved via the notion of a cross section.

\begin{definition}
  Let \( \mathbb{R}^{d} \acts X \) be a Borel flow.  A \emph{cross section} for the flow is a Borel set
  \( \mathcal{C} \subseteq X \) that intersects every orbit in a non-empty
  lacunary\footnote{Sometimes the definition is weakened by required that the intersection with
    every orbit is countable, but since lacunary sections always exist, there is no
    harm in adopting the stronger notion.} set, i.e., \( \mathcal{C} \cap [x]_{\cber[X]} \ne \es \)
  for all \( x \in X \), and there is a non-empty neighborhood of the identity
  \( U \subseteq \mathbb{R}^{d} \) such that \( (x + U) \cap \mathcal{C}  = \{x\} \) for all \( x \in \mathcal{C} \).
  When one wants to specify \( U \) explicitly, \( \mathcal{C} \) is called \( U \)-lacunary.
\end{definition}

One can think of a cross section as being a discrete version of the ambient equivalence relation.  A
theorem of A.~S.~Kechris \cite{MR1176624} shows that all Borel flows admit cross sections.  The
notion of a cross section can be further strengthened by requiring cocompactness.

\begin{definition}
  A cross section \( \mathcal{C} \) for a flow \( \mathbb{R}^{d} \acts X \) is \emph{cocompact} if there
  exists a compact set \( K \subseteq \mathbb{R}^{d} \) such that \( \mathcal{C} + K = X \).
\end{definition}

C.~Conley proved existence of cocompact cross sections for all Borel flows (see, for instance,
\cite[Section 2]{2015arXiv150400958S}).  Our approach to separate local and global aspects of
time change equivalence starts with cocompact cross sections \( \mathcal{C}_{i} \subseteq X_{i} \)
for two flows \( \mathbb{R}^{d} \acts X_{i} \), \( i = 1,2 \), and a \emph{given} orbit equivalence
\( \phi : \mathcal{C}_{1} \to \mathcal{C}_{2} \).  The question is then whether \( \phi \) can be
extended to a \emph{time change equivalence} \( \phi : X_{1} \to X_{2} \).  Since \( \phi \) is
given on points from \( \mathcal{C}_{1} \), its global behavior is uniquely defined, and one can
concentrate on the local aspect of the problem.  This approach is already implicit in the
aforementioned work of Feldman \cite{MR1113569}.  Before stating our results, we need one more
definition.

\begin{definition}
  A Borel flow \( \mathbb{R}^{d} \acts X \) is said to be \emph{compressible} if it admits no
  invariant Borel probability measures.  An invariant Borel subset \( Z \subseteq X \) is
  \emph{compressible} if the restriction of the flow onto \( Z \) is compressible.  An invariant set
  \( Z \subseteq X \) is \emph{cocompressible} if its complement is compressible.
\end{definition}

The definition of compressibility given above is concise, but not very useful.  The term
``compressible'' is explained by an important characterization due to M.~G.~Nadkarni
\cite{MR1081705} (see also \cite[Theorem 4.3.1]{MR1425877}) of the direct analog of this notion for countable
equivalence relations.

Our first result in this paper is the following theorem.

\begin{theorem}[see Theorem \ref{thm:extending-cross-section-to-time-change}]
  \label{thm:main-theorem}
  Let \( \mathbb{R}^{d} \acts X_{1} \) and \( \mathbb{R}^{d} \acts X_{2} \), \( d \ge 2 \), be free non smooth Borel flows
  and let \( \mathcal{D}_{i} \subseteq X_{i} \) be cocompact cross
  sections.  For any orbit equivalence \( \phi : \mathcal{D}_{1} \to \mathcal{D}_{2} \) there are
  cocompressible invariant Borel subsets \( Z_{i} \subseteq X_{i} \) and a time
  change equivalence \( \psi : Z_{1} \to Z_{2} \) which extends \( \phi \) on \( \mathcal{D}_{1}
  \cap Z_{1}\).
\end{theorem}

A corollary of the theorem above and of the classification of hyperfinite equivalence relations
\cite{MR1149121} is time change equivalence of Borel \( \mathbb{R}^{d} \)-flows up to a compressible
piece.

\begin{theorem}[see Theorem \ref{thm:time-change-equivalence-by-compressible}]
  \label{thm:time-change-equivalence-on-cocompressible}
  Let \( \mathbb{R}^{d} \acts X_{i} \), \( i=1,2 \), \( d \ge 2 \), be free non smooth Borel flows.
  There are cocompressible invariant Borel sets \( Z_{i} \subseteq X \) such that the restrictions
  of flows onto these sets are time change equivalent.
\end{theorem}

In Section \ref{sec:periodic-flows} we consider \( \mathbb{R}^{d} \)-flows that are not necessarily
free.  The main results therein are Theorem \ref{thm:selection-theorem} and Theorem
\ref{thm:RxT-flow-tce-to-product}.  The first one shows that one can identify any \( \mathbb{R}^{d}
\)-flow with a number of \emph{free} \( \mathbb{R}^{p} \times \mathbb{T}^{q} \)-flows, and the
latter theorem establishes an analog of Theorem \ref{thm:extending-cross-section-to-time-change} in
this more general context.

Finally, the last section, contains a remark on the difference between time change equivalence and
Lebesgue orbit equivalence, which is defined as an orbit equivalence that preserves the Lebesgue
measure between orbits.  It illustrates the high complexity of Lebesgue orbit equivalence even in the
simplest case of periodic \( \mathbb{R} \)-flows.  

\section{Rational Grids}
\label{sec:rational-grids}

This section provides some technical constructions that will be used in
Section~\ref{sec:back-forth-constr}.  The main concept here is that of a rational grid which will
provide a rigorous justification for why the back-and-forth method in Section
\ref{sec:back-forth-constr} can be performed in a Borel way.

Let \( \mathbb{R}^{d} \acts X \) be a Borel flow.  A \emph{spiral of cross sections}
\( (\mathcal{C}_{n}, h_{n}) \), \(n \in \mathbb{N} \), is a sequence of cross sections
\( \mathcal{C}_{n} \) together with Borel maps \( h_{n} : \mathcal{C}_{n} \to \mathbb{R}^{d} \) such that
\( \mathcal{C}_{n+1} = \mathcal{C}_{n} + h_{n} \) for all \( n \in \mathbb{N} \), i.e.,
 \[ \mathcal{C}_{n+1} = \{ x + h_{n}(x) : x \in
   \mathcal{C}_{n}\}. \]
 With a spiral of cross sections we associate \emph{homomorphism} maps
 \[ \phi_{n,n+1} : \mathcal{C}_{n} \to \mathcal{C}_{n+1}\quad \textrm{given by} \quad \phi_{n,n+1}(x) =  x + h_{n}(x),  \]
 and
 \( \phi_{k,n} : \mathcal{C}_{k} \to \mathcal{C}_{n} \) for \( k \le n \) defined as
 \[ \phi_{k,n} = \phi_{n-1,n} \circ \phi_{n-2,n-1} \circ \cdots \circ \phi_{k,k+1} \]
 with the agreement that \( \phi_{n,n} : \mathcal{C}_{n} \to \mathcal{C}_{n} \) is the identity
 map.  Note that 
 \[ \phi_{m,n} \circ \phi_{k,m} = \phi_{k,n} \quad \textrm{for all } k \le m \le n. \]

 When a  flow \( \flow \) on \( X \) is free, one has a \emph{cocycle}
 \( \rho : \cber[X] \to \mathbb{R}^{d} \) that assigns to a pair \( (x,y) \in \cber[X] \) the unique
 vector \( \vec{r} \in \mathbb{R}^{d} \) such that \( x + \vec{r} = y \).  If
 \( (\mathcal{C}_{n}, h_{n}) \) is a spiral of cross sections for \( \flow \), then for all
 \( x \in \mathcal{C}_{0} \)
 \begin{equation}
   \rho\bigl(x,\phi_{0,n}(x)\bigr) = h_{0}(x) + h_{1}\bigl
   (\phi_{0,1}(x)\bigr) + \cdots +
   h_{n-2}\bigl(\phi_{0,n-2}(x)\bigr) + h_{n-1}\bigl(\phi_{0,n-1}(x)\bigr).
   \tag{\dag}
   \label{eq:cocycle-estimate}
 \end{equation}
 
 A spiral of cross sections \( (\mathcal{C}_{n}, h_{n}) \), \( n \in \mathbb{N} \), of a free flow
 is said to be \emph{convergent} if for all \( x \in \mathcal{C}_{0} \) the limit
 \( \lim_{n} \rho\bigl(x, \phi_{0,n}(x)\bigr) \) exists.  For a convergent spiral we define the
 \emph{limit shift} maps \( H_{k} : \mathcal{C}_{k} \to \mathbb{R}^{d} \) by
 \[ H_{k}(x) = \lim_{n \to \infty} \Bigl[ h_{k}(x) + h_{k+1}\bigl(\phi_{k,k+1}(x)\bigr) + \cdots +
   h_{n-2}\bigl(\phi_{k,n-2}(x)\bigr) + h_{n-1}\bigl(\phi_{k,n-1}(x)\bigr)\Bigr].  \]
 Being a pointwise limit of Borel functions, \( H_{k} \) is Borel.  The \emph{limit cross section} of a
 convergent spiral is a set \( \mathcal{D} \subseteq X \) defined by
 \[ \mathcal{D} = \mathcal{C}_{0} + H_{0} = \{ x + H_{0}(x) : x \in \mathcal{C}_{0} \}. \]
 Note that \( \mathcal{D} = \mathcal{C}_{k} + H_{k} \) for any \( k \in \mathbb{N} \).  Also, we let
 \( \phi_{k, \infty} : \mathcal{C}_{k} \to \mathcal{D} \) to be
 \( \phi_{k,\infty}(x) = x + H_{k}(x) \).  The set \( \mathcal{D} \) is necessarily Borel, as it is
 a countable-to-one image of a Borel function.  In general, \( \mathcal{D} \) may not be lacunary,
 but the following easy conditions guarantee lacunarity of the limit cross section.  Hereafter
 \( B(\delta) \subseteq \mathbb{R}^{d} \) denotes an open ball of radius \( \delta \) around the
 origin.

 \begin{proposition}
   \label{prop:lacunarity-of-limit-cross-section}
   Let \( \mathfrak{C} = (\mathcal{C}_{n}, h_{n}) \), \( n \in \mathbb{N} \), be a spiral of cross
   sections for a free flow \( \mathbb{R}^{d} \acts X \).  If there exists a convergent series
   \( \sum_{i=0}^{\infty}a_{i} \) of positive reals such that
   \( h_{n} : \mathcal{C}_{n} \to B(a_{i}) \), then the spiral \( \mathfrak{C} \) is convergent.  If
   furthermore \( \mathcal{C}_{0} \) is \( B(\delta) \)-lacunary for some \( \delta \) such that
   \( \sum_{i}a_{i} < \delta \), then the limit cross section of the spiral is
   \( B\bigl(\delta - \sum_{i}a_{i}\bigr) \)-lacunary.
 \end{proposition}

 \begin{proof}
   The proof is immediate from the equation~\eqref{eq:cocycle-estimate}.
 \end{proof}



 \begin{definition}
   A \emph{rational grid} for a flow \( \mathbb{R}^{d} \acts X \) is a Borel subset \( Y \subseteq X
   \) which is invariant under the action of \( \mathbb{Q}^{d} \) and intersects every orbit of the flow in
   a unique \( \mathbb{Q}^{d} \)-orbit: for every \( x \in X \) there is \( y \in Y \) such that \(
   [x]_{\mathbb{R}^{d}} \cap Y = [y]_{\mathbb{Q}^{d}} \).
\end{definition}

\begin{lemma}
  \label{lem:existence-of-rational-grids}
  Any free Borel flow admits a rational grid.
\end{lemma}

\begin{proof}
  For every \( \epsilon > 0 \) we pick a Borel map
  \( \alpha_{\epsilon} : \mathbb{R}^{d} \to \mathbb{Q}^{d} \) such that
  \( \bigl|\bigl| \vec{r} - \alpha_{\epsilon}(\vec{r}) \bigr|\bigr| < \epsilon \) for all
  \( \vec{r} \in \mathbb{R}^{d} \).  By a theorem of Kechris \cite{MR1176624}, there exists a
  \( B(2) \)-lacunary cross section \( \mathcal{C} \subseteq X \).  Restriction
  \( \cber[\mathcal{C}] \) of the orbit equivalence relation onto \( \mathcal{C} \) is hyperfinite
  (see \cite[Theorem 1.16]{MR1900547}), and one may therefore represent \( \cber[\mathcal{C}] \) as
  an increasing union of finite equivalence relations: \( \cber[\mathcal{C}] = \bigcup_{n} \fber[m] \).

  We are going to construct a spiral of cross sections \( (C_{n}, h_{n}) \), \( n \in \mathbb{N} \),
  \( \mathcal{C}_{0} = \mathcal{C} \), such that
  \( h_{n} : \mathcal{C}_{n} \to B\bigl(2^{-n-1}\bigr) \).  Let \( \fber[m][n] \) denote the
  equivalence relation \( \fber[m] \) transferred onto \( \mathcal{C}_{n} \) via \( \phi_{0,n} \):
  \[ x \fber*[m][n] y \iff \phi^{-1}_{0,n}(x) \fber*[m] \phi^{-1}_{0,n}(y). \]
  The spiral will satisfy the following two conditions:
  \begin{enumerate}
  \item\label{item:hn-constant-on-Cn-classes} \( h_{n} \) is constant on
    \( \fber[n-1][n] \)-equivalence classes:
    \[ x \fber*[n-1][n] y \implies h_{n}(x) = h_{n}(y). \]
  \item\label{item:Fn-class-on-rational-grid} Every \( \fber[n-1][n] \) class in \( \mathcal{C}_{n} \)
    is ``on a rational grid:''
    \[ x \fber*[n-1][n] y \implies \rho(x,y) \in \mathbb{Q}^{d}. \]
  \end{enumerate}

  To this end pick a Borel linear ordering \( \prec \) on \( X \).  For the base of construction we set
  \( \mathcal{C}_{0} = \mathcal{C} \); let \( s_{0} : \mathcal{C}_{0} \to \mathcal{C}_{0} \) be the
  Borel selector that picks the \( \prec \)-minimal element within \( \fber[0][0] \)-classes, and
  define \( h_{0} : \mathcal{C}_{0} \to B(1/2) \) to be
  \[ h_{0}(x) = \alpha_{1/2}\bigl(\rho\bigl(s_{0}(x), x\bigr)\bigr) - \rho\bigl(s_{0}(x),x\bigr). \]
  The cross section \( \mathcal{C}_{1} \) is then the \( h_{0} \)-shift of \( \mathcal{C}_{0} \):
  \( \mathcal{C}_{1} = \bigl\{ x + h_{0}(x) : x \in \mathcal{C}_{0} \bigr\} \).  Geometrically,
  \( \mathcal{C}_{1} \) is constructed by shifting points by at most \( 1/2 \) within
  \( \fber[0][0] \)-classes relative to the origin provided by the minimal point \( s_{0}(x) \).

  The inductive step is very similar, with a notable difference lying in the fact that when moving
  points within each \( \fber[n][n] \)-class, together with any point we move its
  \( \fber[n-1][n] \)-class.  More precisely, suppose \( \mathcal{C}_{n}, h_{n-1} \) have been
  constructed, and let \( s_{n} : \mathcal{C}_{n} \to \mathcal{C}_{n} \) be the Borel selector,
  which picks \( \prec \)-minimal points within \( \fber[n][n] \)-classes.  Let also
  \( \tilde{s}_{n-1} : \mathcal{C}_{n} \to \mathcal{C}_{n} \) denote the Borel selector for
  \( \fber[n-1][n] \)-classes.  Define
  \( h_{n} : \mathcal{C}_{n} \to B\bigl(2^{-n-1}\bigr) \) by setting
  \[ h_{n}(x) = \alpha_{2^{-n-1}}\bigl(\rho\bigl(s_{n}(x), \tilde{s}_{n-1}(x)\bigr)\bigr) -
    \rho\bigl(s_{n}(x), \tilde{s}_{n-1}(x)\bigr). \]

  It is easy to see that items \eqref{item:hn-constant-on-Cn-classes} and
  \eqref{item:Fn-class-on-rational-grid} are satisfied.  Also,
  \( h_{n} : \mathcal{C}_{n} \to B(2^{-n-1}) \) and Proposition
  \ref{prop:lacunarity-of-limit-cross-section} ensure that \( (\mathcal{C}_{n}, h_{n}) \) converges,
  and the limit cross section \( \mathcal{D} \) is \( B(1) \)-lacunary.  We claim that
  \( \mathcal{D} \) is on a rational grid in the sense that \( \rho(x,y) \in \mathbb{Q}^{d} \) for
  all \( x, y \in \mathcal{D} \) such that \( x \cber*[X] y \).  Indeed, let \( m \) be so large
  that \( \phi_{0,\infty}^{-1}(x) \fber*[m][0] \phi_{0,\infty}^{-1}(y) \), and therefore also
  \[ \phi_{m+1,\infty}^{-1}(x) \fber*[m][m+1] \phi_{m+1,\infty}^{-1}(y). \]
  Item \eqref{item:Fn-class-on-rational-grid} implies that 
  \[ \vec{q} := \rho\bigl(\phi_{m+1,\infty}^{-1}(x), \phi_{m+1,\infty}^{-1}(y)\bigr) \in
    \mathbb{Q}^{d}. \]
  It now follows from item \eqref{item:hn-constant-on-Cn-classes} and the definition of the limit
  cross section that \( \rho(x,y) = \vec{q} \) and
  thus \( \rho(x,y) \in \mathbb{Q}^{d} \) as claimed.

  The required rational grid is given by \( \mathcal{D} + \mathbb{Q}^{d} \).

\end{proof}




Let \( Q \) be a rational grid for a Borel flow \( \mathbb{R}^{d} \acts X \).  We say that a cross
section \( \mathcal{C} \) is \emph{on the grid}~\( Q \) if \( \mathcal{C} \subseteq Q \).
A small perturbation allows one to shift any cross section to a given grid.

\begin{lemma}
  \label{lem:rational-perturbation}
  Let \( Q \subseteq X \) be a rational grid for a free flow \( \mathbb{R}^{d} \acts X \).  For any
  cross section \( \mathcal{C} \subseteq X \) and any \( \epsilon > 0 \) there exist a cross section
  \( \mathcal{C}' \) on the grid \( Q \) and a Borel orbit equivalence
  \( \phi : \mathcal{C} \to \mathcal{C}' \) such that
  \( \bigl|\bigl|\rho\bigl(x, \phi(x)\bigr)\bigr|\bigr| < \epsilon \) for all \( x \in \mathcal{C} \).
\end{lemma}

\begin{proof}
  Let \( \delta > 0 \) be so small that \( \mathcal{C} \) is \( B(\delta) \)-lacunary.  We may assume
  without loss of generality that \( \epsilon < \delta \).  Let
  \( P \subseteq \mathcal{C} \times Q \) be the set
  \[ P = \bigl\{ (x, y) \in \mathcal{C} \times Q : x \cber*[X] y \textrm{ and } ||\rho(x,y)|| <
    \epsilon \bigr\}. \]
  Clearly, \( \proj_{\mathcal{C}}(P) = \mathcal{C} \).  Since the projection of \( P \) onto the
  first coordinate is countable-to-one, Luzin--Novikov Theorem (see \cite[18.14]{MR1321597})
  guarantees existence of a Borel ``inverse'', i.e., a Borel map \( \phi : \mathcal{C} \to Q \) such
  that \( \bigl(c, \phi(c)\bigr) \in P \) for all \( c \in \mathcal{C} \).  The map \( \phi \) is
  injective, since \( \epsilon < \delta \).  The required cross section \( \mathcal{C}' \) is given
  by \( \phi(\mathcal{C}) \).
\end{proof}

\section{Some Simple Tools}
\label{sec:some-simple-tools}

In this section we gather a few elementary tools that will be useful in the proof of the main theorem.

\begin{lemma}[Small perturbation lemma]
  \label{lem:small-perturbations-are-time-change-equivalent}
  Let \( \mathbb{R}^{d} \acts X \) be a free Borel flow, let
  \( \mathcal{C}, \mathcal{D} \subseteq X \) be Borel cross sections, and let
  \( \phi : \mathcal{C} \to \mathcal{D} \) be an orbit equivalence.  If \( \mathcal{C} \) is
  \( B(\delta) \)-lacunary for some \( \delta > 0 \) and
  \[ \bigl|\bigl|\rho\bigl(x, \phi(x)\bigr)\bigr|\bigr| < \delta \quad \textrm{for all} \quad x \in
    \mathcal{C}, \]
  then there exists a time change equivalence \( \psi : X \to X \) that extends \( \phi \).
\end{lemma}

\begin{proof}
  For any \( \vec{r} \) in \( B(\delta) \subset \mathbb{R}^{d} \) there exists a diffeomorphism
  \( f_{\vec{r}} : B(\delta) \to B(\delta) \) with compact support such that
  \( f_{\vec{r}}(\vec{0}) = \vec{r} \).  The proof of this assertion is sufficiently concrete (see
  e.g., \cite[p.~22]{MR1487640}), so that the dependence on \( \vec{r} \) is Borel, i.e., one can pick a
  Borel map \( f : B(\delta) \times B(\delta) \to B(\delta) \) such that
  \( f(\vec{r}, \, \cdot \,) : B(\delta) \to B(\delta) \) is a compactly supported diffeomorphism,
  and \( f(\vec{r},\vec{0}) = \vec{r} \) for all \( \vec{r} \in B(\delta) \).

  Let \( \xi : \mathcal{C} + B(\delta) \to \mathcal{C} \) be the map given by
  \( \xi(x + \vec{r}) = x \) for all \( x \in \mathcal{C} \) and \( \vec{r} \in B(\delta) \).  The
  required time change equivalence \( \psi : X \to X \) is defined by the formula:
  \[ \psi(x) =
    \begin{cases}
      x & \textrm{if } x \not \in \mathcal{C} + B(\delta),\\
      \xi(x) + f\bigl(\rho\bigl(\xi(x), \phi \circ \xi(x)\bigr), \rho\bigl(\xi(x),x\bigr)\bigr) &
      \textrm{otherwise}.\\
    \end{cases}
  \]
  The somewhat cryptic definition of \( \psi(x) \) is really simple: within a ball
  \( c + B(\delta) \), \( c \in \mathcal{C} \), we apply the diffeomorphism
  \( f(\vec{r}, \, \cdot \,) \), \( \vec{r} = \rho(c, \phi(c)) \), ensuring that
  \( \psi(c) = \phi(c) \).  By assumption \( \mathcal{C} \) is \( B(\delta) \)-lacunary, and
  therefore \( \psi \) is injective, and hence is a time change equivalence.
\end{proof}

One of the primary tools to construct orbit equivalences is Rokhlin's Lemma.  The following
provides a concrete form that we are going to use.  The statement essentially coincides with that of
Theorem 6.3 in \cite{2015arXiv150400958S} with addition of item \eqref{item:on-the-grid}, which
asserts that cross section \( \mathcal{C}_{n} \) can be taken to be on the given grid \( Q \).  This
modification is straightforward in view of Lemma
\ref{lem:small-perturbations-are-time-change-equivalent} above.

\begin{lemma}
  \label{lem:uniform-rokhlin}
  For any free Borel flow \( \mathbb{R}^{d} \acts X \) and any rational grid \( Q \subseteq X \)
  there exist a Borel cocompressible invariant set \( Z \subseteq X \), a sequence of Borel cross
  sections \( \mathcal{C}_{n} \subseteq Z \), and
  an increasing sequence of positive rationals \( (l_{n})_{n=1}^{\infty} \) such that for rectangles
  \( R_{n}= [-l_{n}, l_{n}]^{d} \) one has:
  \begin{enumerate}[(i)]
  \item \( \lim_{n \to \infty} l_{n} = \infty \);
  \item \( Z = \bigcup_{n} \bigl( \mathcal{C}_{n} + R_{n} \bigr) \);
  \item \( (c + R_{n}) \cap (c' + R_{n}) = \es  \) for all distinct \( c, c' \in R_{n} \);
  \item \( \mathcal{C}_{n} + R_{n} \subseteq \mathcal{C}_{n+1} + R_{n+1}^{\la 1} \), where
    \( R_{n+1}^{\la 1} \) is obtained by shrinking the square \( R_{n+1} \) by \( 1 \) in
    every direction:
    \[ R_{n+1}^{\la 1} = [-l_{n+1} + 1, l_{n+1}-1]^{d}; \]
  \item\label{item:on-the-grid} \( \mathcal{C}_{n} \subseteq Q \).
  \end{enumerate}
\end{lemma}

For any flow \( \mathbb{R}^{d} \acts X \), we let \( \mathcal{E}(X) \) to denote the Borel space of
invariant ergodic probability measures on~\( X \).  The construction of the time change equivalence
given in Section \ref{sec:back-forth-constr} would be easier if performed relative to a fixed
ergodic measure on \( X \).  To make it work generally, we make use of the following classical
ergodic decomposition theorem due to Varadarajan.

\begin{lemma}[Ergodic Decomposition]
  \label{lem:ergodic-decomposition}
  For any free Borel flow \( \mathbb{R}^{d} \acts X \) with \( \mathcal{E}(X) \ne \es \) there
  exists a Borel surjection \( x \mapsto \mu_{x} \) from \( X \) onto  \( \mathcal{E}(X) \)
  such that
  \begin{enumerate}[(i)]
  \item \( x \cber*[X] y \implies \mu_{x} = \mu_{y} \);
  \item \( \nu \bigl(\bigl\{ x : \mu_{x} = \nu \bigr\}\bigr) = 1 \) for any \( \nu \in \mathcal{E}(X) \).
  \end{enumerate}
\end{lemma}

The following extension lemma will be used routinely through the back-and-forth construction.  It is
taken directly from \cite[Proposition 2.6]{MR1113569}.

\begin{lemma}[Extension Lemma]
  \label{lem:extension-lemma}
  Let \( R, R' \) and \(D_{i}, D_{i}' \subseteq \mathbb{R}^{d} \), \( 1 \le i \le n \), be smooth
  disks such that \( D_{i} \subseteq R \) and \( D_{i}' \subseteq R' \).  Any family of orientation preserving
  diffeomorphisms \( f_{i} : D_{i} \to D_{i}' \) admits a common extension to an orientation
  preserving diffeomorphism \( f : R \to R' \).  
\end{lemma}

Finally, we shall need the following easy fact from the theory of countable Borel equivalence relations.

\begin{lemma}
  \label{lem:smooth-partition}
  Let \( \cber, \fber \) be finite Borel equivalence relation on a standard Borel space \( X \).
  Suppose that \( \fber \subseteq \cber \).  There is a sequence of disjoint Borel sets \( A_{n}
  \subseteq X \) such that
  \begin{enumerate}
  \item \( X = \bigsqcup_{n} A_{n} \);
  \item each \( A_{n} \) is \( \fber \)-invariant;
  \item \( [x]_{\cber} \cap A_{n} = [x]_{\fber} \) for all \( x \in A_{n} \).
  \end{enumerate}
  If there is a bound on \( [\fber:\cber] \) --- the number of \( \fber \)-classes in a \( \cber
  \)-class --- then the sequence \( (A_{n})_{n} \) can be taken to be finite.
\end{lemma}

\begin{proof}
  Since \( \cber \) is smooth, it admits a Borel transversal \( B_{0} \subseteq X \).  Set \( A_{0}
  = [B_{0}]_{\fber} \), and set recursively \( A_{n} = [B_{n}]_{\fber} \), where \( B_{n} \)
  is a Borel transversal for \( \cber \) restricted onto \( X \setminus \bigcup_{k< n} A_{k} \). 
\end{proof}

\section{Back and forth construction}
\label{sec:back-forth-constr}

For the proof of the following theorem it is convenient to introduce the notion of a tree of
partitions.  Let \( \mathbb{R}^{d} \acts X \) be a free flow on \( X \).  A \emph{tree of
  partitions} for the flow is a family of invariant Borel sets
\( (\Omega_{s})_{s \in \mathbb{N}^{<\mathbb{N}}} \), \( \Omega_{s} \subseteq X \), indexed by finite
sequences of natural numbers that satisfies the following two conditions:
\begin{enumerate}
\item \( X = \bigsqcup_{s \in \mathbb{N}^{n}} \Omega_{s} \) for each \( n \in \mathbb{N} \); in
  particular, \( \Omega_{\es} = X \).
\item \( s \subseteq t \implies \Omega_{t} \subseteq \Omega_{s} \). 
\end{enumerate}

\begin{theorem}
  \label{thm:tce-up-to-cocompressible}
  Let \( \mathbb{R}^{d} \acts X_{1} \) and \( \mathbb{R}^{d} \acts X_{2} \) be free Borel flows on
  standard Borel spaces, let for \( i=1,2 \), \( Q_{i} \subseteq X_{i} \) be rational grids, let
  \( \mathcal{D}_{i} \subseteq Q_{i} \) be cocompact cross sections on these grids, and let
  \( \zeta : \mathcal{D}_{1} \to \mathcal{D}_{2} \) be an orbit equivalence between them.  There are
  Borel invariant cocompressible sets \( Z_{i} \subseteq X_{i} \), and a time change equivalence
  \( \psi : Z_{1} \to Z_{2} \) that extends \( \zeta \), i.e., 
  \( \zeta|_{Z_{1} \cap \mathcal{D}_{1}} = \psi|_{Z_{1} \cap \mathcal{D}_{1}} \).
\end{theorem}

\begin{proof}
  The proof relies on a back-and-forth argument similar to the one used in the proof of Theorem 1 in
  \cite{MR1113569}.
  For start, let us apply the Uniform Rokhlin Lemma (Lemma \ref{lem:uniform-rokhlin}) to both flows
  yielding Borel invariant cocompressible sets \( \tilde{Z}_{i} \subseteq X_{i} \), as well as cross
  sections \( \mathcal{C}_{i,n} \subseteq \tilde{Z}_{i} \) , \( n \in \mathbb{N} \), and rationals
  \( l_{i,n} \in \mathbb{Q} \) such that for the squares
  \( R_{i,n} = [-l_{i,n}, l_{i,n}]^{d} \) , \( i=1,2 \), one has
  \begin{enumerate}
  \item \label{item:on-rational-grid} cross sections \( \mathcal{C}_{i,n} \subseteq Q_{i} \) 
   are on the rational grids;
 \item \label{item:exhaust-orbits}
   \( \tilde{Z}_{i} = \bigcup_{n} \bigl( \mathcal{C}_{i,n} + R_{i,n}\bigr) \);
 \item boxes \( c + R_{i,n} \) are pairwise disjoint;
  \item \label{item:nested} \( \mathcal{C}_{i,n} + R_{i,n} \subseteq \mathcal{C}_{i,n+1} +
    R_{i,n+1}^{\la 1} \);
  \item \label{item:sizes-increase} sequences \( (l_{i,n})_{n \in \mathbb{N}} \) are increasing and
    unbounded.
  \end{enumerate}
  Since cross sections \( \mathcal{D}_{i} \) are cocompact, we may omit, if necessary, finitely many
  cross sections \( \mathcal{C}_{i,n} \)  and assume without loss of generality that \( l_{i,n} \) are so
  large that \( \mathcal{D}_{i} \cap (c + R_{i,n}) \ne \es \) for all \( c \in \mathcal{C}_{i,n} \).
  We shall further decrease sets \( \tilde{Z}_{i} \) by throwing away invariant compressible
  sets, so for notational convenience we assume that \( \tilde{Z}_{i} = X_{i} \).

  For each of \( X_{i} \) we pick an ergodic decomposition \( x \mapsto \mu_{x} \) as in
  Lemma \ref{lem:ergodic-decomposition}.  We also let \( \fber[n][i] \) to denote finite Borel
  equivalence relations on \( \mathcal{D}_{i} \cap \bigl(\mathcal{C}_{i,n} + R_{i,n} \bigr) \) given by
  \[ x \fber*[n][i] y \iff x, y \in (c + R_{i,n}) \quad \textrm{for some } c \in
    \mathcal{C}_{i,n}. \]
  Each \( \fber*[n][i] \) class lives in a unique \( R_{i,n} \) box.  
  We are going to construct trees of Borel partitions
  \( (\Omega_{i,s})_{s \in \mathbb{N}^{<\mathbb{N}}} \) on \( X_{i} \), together with families of positive
  integers \( (\omega_{i,s})_{s \in \mathbb{N}^{<\mathbb{N}}} \).  Before listing properties of
  these objects, let us introduce the following sets:
  \begin{displaymath}
    \begin{aligned}
      \mathcal{V}_{1,s} = \bigl\{ &x \in \mathcal{C}_{1, \omega_{1,s}} : \textrm{for all }
      y_{1},y_{2} \in \mathcal{D}_{1} \cap \bigl( x + R_{1,\omega_{1,s}}\bigr)
      \\ &\textrm{one has } \zeta(y_{1}) {\fber*[\omega_{2,s}][2]} \zeta(y_{2}) \bigr\},\\
      \mathcal{V}_{2,s} = \bigl\{ &x \in \mathcal{C}_{2, \omega_{2,s}} : \textrm{for all }
      y_{1},y_{2} \in \mathcal{D}_{2} \cap \bigl( x + R_{2,\omega_{2,s}}\bigr)
      \\ &\textrm{one has } \zeta^{-1}(y_{1}) {\fber*[\omega_{1,s}][1]} \zeta^{-1}(y_{2}) \bigr\}.\\
    \end{aligned}
  \end{displaymath}
  Figure \ref{fig:sets-VXn} illustrates the definition of the set \( \mathcal{V}_{1,s} \): a point
  \( x \in \mathcal{C}_{1,\omega_{1,s}}\) belongs to \( \mathcal{V}_{1,s} \) if the images under
  \( \zeta \) of all the points of \( \mathcal{D}_{1} \) in the box \( R_{1,\omega_{1,s}} \) around
  \( x \) fall into a single box \( R_{2,\omega_{2,s}} \) in \( X_{2} \).  The definition of
  \( \mathcal{V}_{2,s} \) uses \( \zeta^{-1} \) instead of \( \zeta \).  
  \begin{figure}[htb]
    \centering
    \begin{tikzpicture}
      \draw[fill=black!15!white] (-1, -1) rectangle (0.5,0.5);
      \filldraw (-0.7, -0.5) circle (1pt) node[anchor=north] {\( x_{1} \)};
      \filldraw (-0.3, -0.1) circle (1pt) node[anchor=south] {\( x_{2} \)};
      \filldraw (0.25, -0.4) circle (1pt) node[anchor=north] {\( x_{3} \)};
      \draw[fill=black!15!white] (1, -1.5) rectangle (2.5,0);
      \filldraw (1.8, -0.6) circle (1pt) node[anchor=east] {\( y_{1} \)};
      \filldraw (2.1, -1.2) circle (1pt) node[anchor=east] {\( y_{2} \)};
      \draw (0.8, 0.4) rectangle (2.3,1.9);
      \filldraw (1.3,1) circle (1pt) node[anchor=north] {\( z_{1} \)};
      \filldraw (1.9, 1.4) circle (1pt) node[anchor=north] {\( z_{2} \)};
      \filldraw (0,0.9) circle (1pt);
      \filldraw (-0.5, 1.2) circle (1pt);
      \filldraw (-0.1, 1.6) circle (1pt);
      \filldraw (-0.1, -1.3) circle (1pt);
      \draw[->] (3,0.25) -- (4.5,0.25) node[pos=0.5,anchor=south] {\( \zeta \)};
      \draw (5, -1.25) rectangle (8, 1.75);
      \filldraw (5.3, -0.9) circle (1pt) node[anchor=west] {\( \zeta(x_{1}) \)};
      \filldraw (5.5, -0.3) circle (1pt) node[anchor=west] {\( \zeta(y_{2}) \)};
      \filldraw (6.9, -0.7) circle (1pt) node[anchor=west] {\( \zeta(x_{3}) \)};
      \filldraw (6.7, 1.3) circle (1pt) node[anchor=west] {\( \zeta(x_{2}) \)};
      \filldraw (5.4, 0.6) circle (1pt) node[anchor=west] {\( \zeta(y_{1}) \)};
      \filldraw (6.5, 0.1) circle (1pt) node[anchor=west] {\( \zeta(z_{1}) \)};
      \filldraw (8.5, 0.2) circle (1pt) node[anchor=north] {\( \zeta(z_{2}) \)};
    \end{tikzpicture}
    \caption{Definition of sets \( \mathcal{V}_{i,s} \).}
    \label{fig:sets-VXn}
  \end{figure}
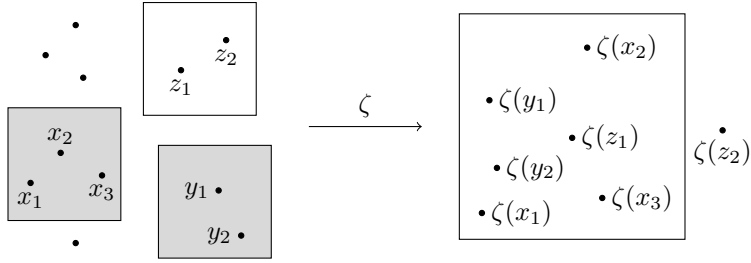  
  The role of integers \( \omega_{i,s} \) will be to ensure that sets \( \mathcal{V}_{i,s} \) are
  sufficiently large in measure.  

  We are now ready to list the conditions on the trees of Borel partitions \( \Omega_{i,s} \) and
  natural \( \omega_{i,s} \).
  \begin{enumerate}
  \item\label{item:invariant-erg-decomp} Sets \( \Omega_{i,s} \) are invariant with respect to the
    ergodic decompositions, i.e., \( \mu_{x} = \mu_{y} \) and \( x \in \Omega_{i,s} \) implies
    \( y \in \Omega_{i,s} \).
  \item\label{item:omega-grow} \( \omega_{i,s} \ge |s| \);
  \item\label{item:omega-monotone} \( s \subseteq t \) implies \( \omega_{i,s} \le \omega_{i,t} \);
  \item\label{item:large-measure} If \( |s| \) is odd, then for any \( x \in \Omega_{1,s} \) one has
    \( \mu_{x}\bigl(\mathcal{V}_{1,s} + R_{1,\omega_{1,s}}\bigr) \le 2^{-|s|} \); if \( |s| \) is even, then
    \[ \mu_{x}\bigl(\mathcal{V}_{2,s} + R_{2,\omega_{2,s}}\bigr) \le 2^{-|s|} \textrm{ for all } x
      \in \Omega_{2,s} .\]
  \end{enumerate}



  Let us first finish the argument under the assumption that such objects have been constructed.
  The base of the inductive construction is the map
  \[ \psi_{1} : \bigcup_{s \in \mathbb{N}^{1}} \bigl(\mathcal{V}_{1,s} + R_{1,\omega_{1,s}}\bigr) \to
    X_{2},\]
  which will be an orientation preserving diffeomorphism between orbits on its domain.  
  Pick some \( s \in \mathbb{N}^{1} \), and define \( \psi_{1} \) on \( \mathcal{V}_{1,s} +
  R_{1,\omega_{1,s}} \) as follows.

  \begin{figure}[htb]
    \centering
    \includegraphics{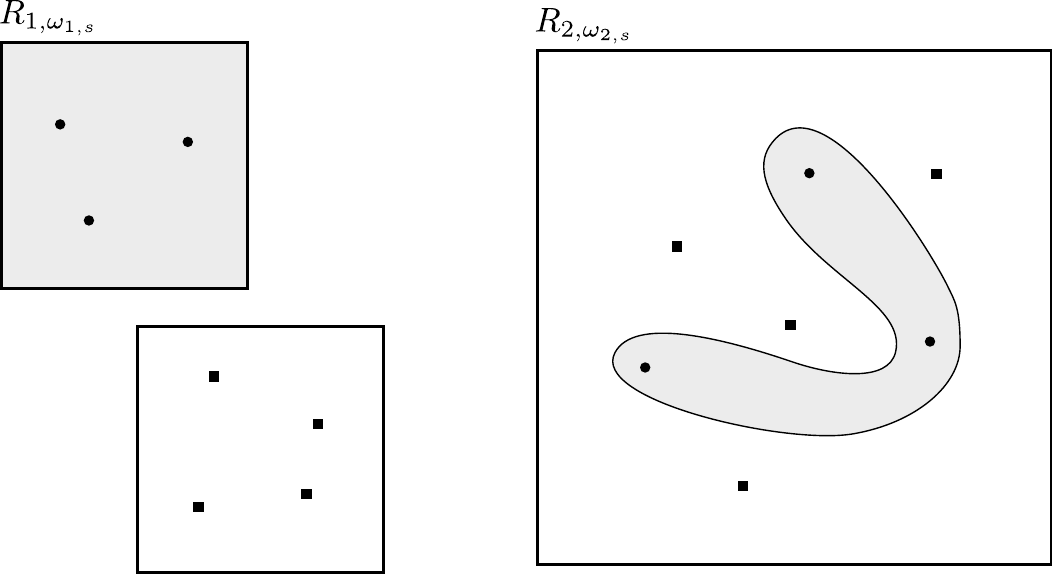}
    \caption{Extension step}
    \label{fig:extension-step-one}
  \end{figure}

  To a point \( x \in \mathcal{V}_{1,s} \) there corresponds a box \( x + R_{1,\omega_{1,s}} \),
  marked gray in Figure \ref{fig:extension-step-one}, which contains several points, say
  \( y_{1}, \ldots, y_{m} \in \mathcal{D}_{1} \).  Images of these points,
  \( \zeta(y_{1}), \ldots, \zeta(y_{m}) \), fall into a single \( z + R_{2,\omega_{2,s}} \) box,
  \( z \in \mathcal{C}_{2,\omega_{2,s}} \).  Besides points
  \( \zeta(y_{1}), \ldots, \zeta(y_{m}) \), the box \( z + R_{2,\omega_{2,s}} \) may contain other
  points of \( \mathcal{D}_{2} \).  We pick any smooth disk inside \( z + R_{2,\omega_{2,s}} \) that
  contains all the points \( \zeta(y_{1}), \ldots, \zeta(y_{m}) \) and does not contain any other points
  of \( \mathcal{D}_{2} \).  One now would like to extend the map
  \[ \zeta : \bigl\{ y_{1}, \ldots, y_{m}\bigr\} \to \bigl\{\zeta(y_{1}), \ldots,
    \zeta(y_{m})\bigr\} \]
  to an orientation preserving diffeomorphism \( \psi_{1} \) from \( x + R_{1,\omega_{1,s}} \) to
  the smooth disk around \( \zeta(y_{1}), \ldots, z(y_{m}) \).  This can be done by the Extension
  Lemma \ref{lem:extension-lemma} (the fact that \( \zeta \) is defined on points rather than disks is, of
  course, immaterial, as is the fact that \( R_{1,\omega_{1,s}} \) is not a smooth disk, since it has
  corners; to be pedantic, one extends \( \zeta \) to little balls around points in
  \( \mathcal{D}_{i} \) in a linear fashion, and considers
  \( \tilde{R}_{i,\omega_{i,s}} \subseteq R_{i,\omega_{i,s}} \) --- ``rectangles with smoothed
  corners'' instead).  The problem that arises with the use of Extension Lemma is the following one.
  The construction needs to be performed in a Borel way, meaning that extension \( \psi_{1} \) has to be defined
  for boxes \( x + R_{1,\omega_{1,s}} \) for all \( x \in \mathcal{V}_{1,s} \) at the same time,
  which can possibly lead to ``collisions'' and prevent \( \psi_{1} \) from being injective.  For
  instance, in Figure \ref{fig:extension-step-one} there are two distinct \( R_{1,\omega_{1,s}} \)
  boxes that must be mapped into a single \( R_{2,\omega_{2,s}} \) box, so we need to ensure that
  their images are disjoint.  The way to do this is to partition \( \mathcal{V}_{1,s} \) into finitely
  many Borel pieces \( \mathcal{V}_{1,s} = A_{1} \sqcup \cdots \sqcup A_{p} \) such that on each
  \( A_{j} \) every \( R_{1,\omega_{1,s}} \) box corresponds to a \emph{unique} square
  \( R_{2,\omega_{2,s}} \) via \( \zeta \).  To this end consider an equivalence relation
  \( \cber \) on \( \mathcal{D}_{1} \cap \bigl(\mathcal{V}_{1,s} + R_{1,\omega_{1,s}}\bigr) \) given
  by
  \[ x \cber* y \iff \zeta(x) \fber*[\omega_{2,s}][2] \zeta(y), \]
  and let \( \fber \) denote the restriction of \( \fber[\omega_{1,s}][1] \) onto \( \mathcal{D}_{1}
  \cap \bigl(\mathcal{V}_{1,s} + R_{1,\omega_{1,s}}\bigr)\).  By the definition of \( \mathcal{V}_{1,s} \) one
  has \( \fber \subseteq \cber \), so Lemma \ref{lem:smooth-partition} applies, and gives a
  partition 
  \[ \mathcal{D}_{1} \cap \bigl(\mathcal{V}_{1,s} + R_{1,\omega_{1,s}}\bigr) =
    \bigsqcup_{j=1}^{p}A_{j}'. \]
  The required partition of \( \mathcal{V}_{1,s} \) is obtained by setting 
  \[
    \begin{aligned}
      A_{j} = \{ &x \in \mathcal{V}_{1,s} : y \in A_{j}' \textrm{ for some (equivalently, any) } y
      \in \mathcal{D}_{1}\\ &\textrm{such that } y \in x + R_{1,\omega_{1,s}}\}.
    \end{aligned}
\]

  \begin{figure}[htb]
    \centering
    \includegraphics{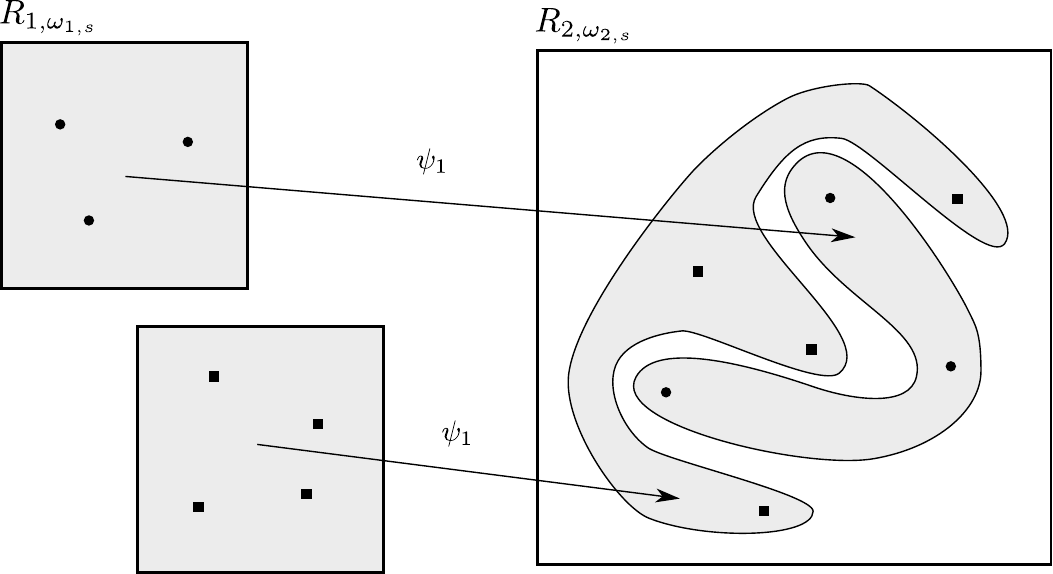}
    \caption{Extension }
    \label{fig:extension-step-two}
  \end{figure}

  We can now define the extension \( \psi_{1} \) with domain
  \( A_{1} + R_{1,\omega_{1,s}} \) as explained above, which is guaranteed to be injective.  Next we
  extend \( \psi_{1} \) to \( A_{2} + R_{1,\omega_{1,s}} \) in a similar way (Figure
  \ref{fig:extension-step-two}).  Given \( x' \in A_{2} \) and points \( y'_{1}, \ldots, y'_{q} \in
  \mathcal{D}_{1} \cap \bigl(x' + R_{1,\omega_{1,s}}\bigr) \), let \( z' \in
  \mathcal{C}_{2,\omega_{2,s}} \) be such that \( \zeta(y'_{j}) \in z' + R_{2,\omega_{2,s}} \) for
  all \( j \).  Pick a smooth disk inside \( z' + R_{2,\omega_{2,s}} \) that contains all the points
  \( \zeta(y'_{j}) \), does not contain any other points of \( \mathcal{D}_{2} \), and does not
  intersect any smooth disks inside \( z' + R_{2,\omega_{2,s}} \) picked at the previous step.  One
  may now use the Extension Lemma to extend \( \psi_{1} \) to a map
  \[ \psi_{1} : (A_{1} \cup A_{2}) + R_{1,\omega_{1,s}} \to \mathcal{C}_{2,\omega_{2,s}} +
    R_{2,\omega_{2,s}}.\]
  Continuing in the same fashion, \( \psi_{1} \) can be extended to all of \( \mathcal{V}_{1,s} +
  R_{1,\omega_{1,s}} \).

  The construction above was performed for a fixed \( s \in \mathbb{N}^{1} \), doing it for all \( s
  \in \mathbb{N}^{1} \) results in the required map 
  \[ \psi_{1} : \bigcup_{s \in \mathbb{N}^{1}} \bigl(\mathcal{V}_{1,s} + R_{1,\omega_{1,s}}\bigr)
    \to \bigcup_{s \in \mathbb{N}^{1}} \bigl(\mathcal{C}_{2,\omega_{2,s}} + R_{2,\omega_{2,s}}\bigr). \]

  We have explained the way \( \psi_{1} \) is defined, but we owe the reader an explanation why this
  construction is Borel.  This is the place where we are going to use rational grids.  Since all
  cross sections \( \mathcal{C}_{i,n} \), \( \mathcal{D}_{i} \) are assumed to be on the rational
  grid \( Q_{i} \), at each step of the construction, every box of the form
  \( x + R_{i,\omega_{i,s}} \), \( x \in \mathcal{C}_{i,n} \), has only \emph{countable many}
  possible configurations.  For example, for any \( x \in A_{1} \) the configuration of
  \( \mathcal{D}_{1} \cap \bigl(x + R_{1,\omega_{1,s}}\bigr) \) is uniquely determined by the
  vectors \( \rho(y_{i}, x) \in \mathbb{Q}^{d} \).  Since we have
  only countably many possible configurations, we can partition \( A_{1} \) into countably many
  pieces by collecting points with the same configuration of boxes around them, and apply \emph{the same}
  extension of \( \psi_{1} \) on each element of this partition.  Such an operation is clearly Borel
  for any choice of smooth disks around points \( y_{1}, \ldots, y_{m} \), any choice of smooth
  extensions given by Lemma \ref{lem:extension-lemma}, etc.  Thus having only countably many cases at
  each step of the back-and-forth construction ensures Borelness.

  We are now done with the base step of the construction.  The inductive step is of little
  difference.  At step \( n \) for an odd value of \( n \), \( \psi_{n} \) is constructed such that
  \( \mathcal{V}_{1,s} + R_{1,\omega_{1,s}}\) is in the domain of \( \psi_{n} \) for all
  \( s \in \mathbb{N}^{n} \), and on even stages we work with \( \zeta^{-1} \), ensuring that
  \( \mathcal{V}_{2,s} + R_{2,\omega_{2,s}} \) is in the range of \( \psi_{n} \) for all
  \( s \in \mathbb{N}^{n} \).  Item \eqref{item:large-measure} in the list of conditions on the sets
  \( \mathcal{V}_{i,s} \) guarantees that sets
  \[
    \begin{aligned}
      Z_{1} &= \bigcup_{m \in \mathbb{N}} \mkern-6mu\bigcap_{|s| \ge m \atop |s| \textrm{ is odd}}\mkern-8mu
      \bigl(\mathcal{V}_{1,s} + R_{1,\omega_{1,s}}\bigr) \\
      Z_{2} &= \bigcup_{m \in \mathbb{N}} \mkern-6mu\bigcap_{|s| \ge m \atop |s| \textrm{ is even}}\mkern-8mu
      \bigl(\mathcal{V}_{2,s} + R_{2,\omega_{2,s}}\bigr) \\
    \end{aligned}
  \]
  have measure one for any invariant ergodic probability measure, and therefore \( X_{i}\setminus
  Z_{i} \) are compressible.  The map \( \psi =
  \bigcup_{n} \psi_{n} \) is the required time change equivalence.
  
  The last remaining bit is to show how the trees of Borel partitions \( (\Omega_{i,s}) \) and integers
  \( (\omega_{i,s}) \) can be constructed.  For the base of construction, \( s = \es \), one sets
  \( \Omega_{i,\es} = X_{i} \) and \( \omega_{i,\es} = 0 \).  Suppose that \( \Omega_{i,s} \) and
  \( \omega_{i,s} \) have been defined for all \( s \in \mathbb{N}^{n} \).  Assume for
  definiteness that \( n \) is even. Pick some \( s \in \mathbb{N}^{n} \).  Since
  \[ \Omega_{1,s} = \bigcup_{n} \Bigl(\bigl( \mathcal{C}_{1,n} \cap \Omega_{1,n}\bigr) +
    R_{1,n}\Bigr), \]
  and the union is increasing, for each \( x \in \Omega_{1,s} \) one may pick \( b \in \mathbb{N} \)
  so large that \( b \ge \max\{\omega_{i,s}, |s| + 1\} \) and
  \[ \mu_{x}\Bigl( \bigl( \mathcal{C}_{1,n} \cap \Omega_{1,n}\bigr) +
    R_{1,n} \Bigr) < 2^{-|s|-2}. \]
  The map that sends \( x \mapsto b(x) \), where \( b(x) \) is the smallest \( b \in \mathbb{N} \)
  that satisfies these conditions is Borel.  Preimages \( W_{m} = b^{-1}(m) \) determine a countable Borel
  partition invariant under the ergodic decomposition:
  \[ \Omega_{1,s} = \bigsqcup_{m} W_{m}. \]
  Considering each \( W_{m} \) separately, we note that for each \( m \)
  \[
    \begin{aligned}
      \bigcup_{n}\Bigl\{ &x \in \mathcal{C}_{1, m} \cap W_{m} : \textrm{for all } y_{1},y_{2} \in
      \mathcal{D}_{1} \cap \bigl( x + R_{1,m}\bigr)\\ &\textrm{one has } \zeta(y_{1}) {\fber*[n][2]}
      \zeta(y_{2}) \Bigr\} = \mathcal{C}_{1,m} \cap W_{m},
    \end{aligned}
  \]
  and therefore for every \( x \in W_{m} \) there is \( n \in \mathbb{N} \) so large that
  \[
    \begin{aligned}
      \mu_{x}\Bigl( \bigl\{ &x \in \mathcal{C}_{1, m} \cap W_{m} : \textrm{for all } y_{1},y_{2} \in
      \mathcal{D}_{1} \cap \bigl( x + R_{1,m}\bigr)\\ &\textrm{one has } \zeta(y_{1}) {\fber*[n][2]}
      \zeta(y_{2}) \bigr\} + R_{1,m} \Bigr) \le 2^{-|s|-1}.
    \end{aligned}
\]
  The map \( x \mapsto c(x) \) that picks the smallest such \( n \) is Borel, and its preimages
  \( c^{-1}(n) \) determine Borel partitions \( W_{m} = \bigsqcup_{n} \tilde{W}_{m,n} \).  By
  re-enumerating \( W_{m,n} \) as \( \Omega_{1, s {}^{\frown} j} \) for \( j \in \mathbb{N} \) we define
  the next level of the tree of partitions.  The corresponding sets \( \Omega_{2,s \frown j} \) are
  defined uniquely by the condition
  \[ \zeta\bigl(\Omega_{1, s {}^{\frown} j} \cap \mathcal{D}_{1}\bigr) = \Omega_{2,s {}^{\frown} j} \cap
    \mathcal{D}_{2}. \]
  Finally, we set \( \omega_{1,s {}^{\frown} j} = m \) if \( \Omega_{1,s {}^{\frown} j} = W_{m,n} \) and we
  set \( \omega_{2,s {}^{\frown} j} = n \) whenever \( \Omega_{1,s {}^{\frown} j} = W_{m,n} \).  This finishes
  the construction of \( (\Omega_{i,s}), (\omega_{i,s}) \) and concludes the proof of the theorem.
\end{proof}

The assumption in the previous theorem that cross sections \( \mathcal{D}_{i} \) lie on rational
grids was used to give an easy argument why the construction of \( \psi_{n} \) is Borel, but the
theorem can now be easily improved by omitting this restriction.

\begin{theorem}
  \label{thm:extending-cross-section-to-time-change}
  Let \( \mathcal{R}^{d} \acts X_{i} \) be free non smooth Borel flows, let \( \mathcal{D}_{i}
  \subseteq X_{i} \) be cocompact cross sections and let \( \zeta : \mathcal{D}_{1} \to
  \mathcal{D}_{2} \) be an orbit equivalence map.  There are cocompressible invariant Borel sets \(
  Z_{i} \subseteq X_{i} \) and a time change equivalence \( \psi: Z_{1} \to Z_{2} \) that extends \(
  \zeta \) on \( Z_{1} \cap \mathcal{D}_{1} \).
\end{theorem}

\begin{proof}
  By Lemma \ref{lem:existence-of-rational-grids} we may pick rational grids
  \( Q_{i} \subseteq X_{i} \).  Lemma \ref{lem:rational-perturbation} allows us to choose cocompact cross
  sections \( \mathcal{D}_{i}' \subseteq Q_{i} \) and Borel orbit equivalence maps \( \phi_{i} :
  \mathcal{D}_{i} \to \mathcal{D}_{i}' \) such that \( \bigl|\bigl| \rho(x, \phi_{i}(x))
  \bigr|\bigr| < \epsilon \) for \( \epsilon \) less than the lacunarity parameter of \(
  \mathcal{D}_{i} \).  By Lemma \ref{lem:small-perturbations-are-time-change-equivalent}, \( \phi_{i}
  \) can be extended to time change equivalences, which we denote by the same letter.  Finally, we
  apply Theorem \ref{thm:tce-up-to-cocompressible} to \( \mathcal{D}_{i}' \subseteq Q_{i} \) and the
  map \( \zeta' : \mathcal{D}_{1}' \to \mathcal{D}_{2}' \) given by
  \[ \zeta'(x) = \phi_{2} \circ \zeta \circ \phi_{1}^{-1} (x), \]
  which produces a time change equivalence \( \psi' : Z_{1} \to Z_{2} \) between cocompressible
  sets.  The required map \( \psi \) is given by \( \psi = \phi_{2}^{-1} \circ \psi' \circ \phi_{1} \).
\end{proof}

\begin{theorem}
  \label{thm:time-change-equivalence-by-compressible}
  Let \( \mathbb{R}^{d} \acts X_{1} \) and \( \mathbb{R}^{d} \acts X_{2} \) be non smooth free Borel
  flows.  There are cocompressible invariant Borel sets \( Z_{i} \subseteq X_{i} \) such that
  restrictions of the flows onto these sets are time change equivalent.
\end{theorem}

\begin{proof}
  We first prove the theorem under the additional assumption that flows posses the same number of
  invariant ergodic probability measures.  Pick cocompact cross sections \( \mathcal{D}_{i}
  \subseteq X_{i} \).  It is known (see, for instance,
  \cite[Proposition 4.4]{2015arXiv150400958S}) that restriction of the orbit equivalence relation onto
  \( \mathcal{D}_{i} \) has the same number of ergodic invariant probability measures as the flow
  \( \mathbb{R}^{d} \acts X_{i} \).  Since orbit equivalence relations on \( \mathcal{D}_{i} \) are
  hyperfinite (by \cite[Theorem 1.16]{MR1900547}), the classification of hyperfinite relations
  \cite[Theorem 9.1]{MR1149121}
  implies that there is an orbit equivalence \( \zeta : \mathcal{D}_{1} \to \mathcal{D}_{2} \).  An
  application of Theorem \ref{thm:extending-cross-section-to-time-change} finishes the argument.

  Since the relation of being time change equivalent up to a compressible set is clearly transitive,
  to complete the proof it is therefore enough to show that for any two possible sizes
  \( \kappa_{1}, \kappa_{2} \in \mathbb{N}^{>0} \cup \{\aleph, \mathfrak{c}\} \) of the spaces of
  ergodic invariant probability measures there are time change equivalent Borel flows
  \( \mathbb{R}^{d} \acts Y_{i} \) with \( \bigl| \mathcal{E}(Y_{i}) \bigr| = \kappa_{i} \).  To
  this end pick Borel \( \mathbb{R} \)-flows \( \mathbb{R} \acts \tilde{Y}_{i} \) 
  such that \( |\mathcal{E}(\tilde{Y}_{i})| = \kappa_{i} \) and let
  \( \mathbb{R}^{d-1} \acts W \) be any uniquely ergodic flow; set
  \( Y_{i} = \tilde{Y}_{i} \times W \) and let \( \mathbb{R}^{d} \acts Y_{i} \) be the product
  action.  One has \( \bigl| \mathcal{E}(Y_{i}) \bigr| = \kappa_{i} \), and we claim
  that these flows are time change equivalent.  By a theorem of B.~Miller and C.~Rosendal
  \cite[Theorem 2.19]{MR2578608}, the flows \( \mathbb{R} \acts \tilde{Y}_{i} \) are time change equivalent via
  some \( \tilde{\phi} : \tilde{Y}_{1} \to \tilde{Y}_{2} \).  Define \( \phi : Y_{1} \to Y_{2} \) by
  the formula
  \[ \phi(y,w) = \bigl(\tilde{\phi}(y), w\bigr).\]
  A straightforward verification shows that \( \phi \) is indeed a time change equivalence between
  the flows \( \mathbb{R}^{d} \acts Y_{1} \) and \( \mathbb{R}^{d} \acts Y_{2} \) as claimed.
\end{proof}

\section{Periodic flows}
\label{sec:periodic-flows}

Recall that for a Polish space \( X \) the \emph{Effros Borel space} of \( X \) is the set \(
\eff(X) \) of closed subsets of \( X \) endowed with the \( \sigma \)-algebra generated by the sets
of the form
\[ \{ F \in \eff(X) : F \cap U \ne \es \}, \qquad U \subseteq X \textrm{ is open}. \]
The space \( \eff(X) \) is a standard Borel space.  We refer the reader to \cite[Sections 12.C,
12.E]{MR1321597} for the basic properties of \( \eff(X) \), one of the main of which is the
Kuratowski--Ryll-Nardzewski Selection Theorem.
\begin{theorem}[Selection Theorem]
  \label{thm:selection-theorem}
  Let \( X \) be a Polish space.  There is a sequence of Borel functions \( f_{n} : \eff(X) \to X \)
  such that for every non-empty \( F \in \eff(X) \) the set \( \{f_{n}(F)\}_{n \in \mathbb{N}} \) is
  a dense subset of \( F \):
  \[ F = \overline{\{f_{n}(F) : n \in \mathbb{N}\}}. \]
\end{theorem}

When \( X \) is a Polish group, one may consider the subset \( \sgr(X) \subseteq \eff(X) \) of
closed subgroups of \( X \), which is a Borel subset of \( \eff(X) \), and is therefore a standard
Borel space in its own right.

In the following we consider the space \( \sgr\bigl(\mathbb{R}^{d}\bigr) \).  A closed subgroup of
\( \mathbb{R}^{d} \) is isomorphic to a group \( \mathbb{R}^{p} \times \mathbb{Z}^{q} \) for some
\( p, q \in \mathbb{N} \), \( p + q \le d \).  This isomorphism can, in fact, be chosen in a Borel
way throughout \( \sgr\bigl(\mathbb{R}^{d}\bigr) \).  For \( p,q \in \mathbb{N} \) with
\( p + q \le d \) we let
\( \sgr_{p,q}\bigl(\mathbb{R}^{d}\bigr) \subseteq \sgr\bigl( \mathbb{R}^{d} \bigr) \) to denote the
set of groups that are isomorphic to \( \mathbb{R}^{p} \times \mathbb{Z}^{q} \).

\begin{lemma}
  \label{lem:effros-basis-selectors} Let \( \ssp(\mathbb{R}^{d}) \) denote the set of all subspaces
  of \( \mathbb{R}^{d} \).
  \begin{enumerate}
  \item\label{item:span-is-borel} The map \( \eff\bigl(\mathbb{R}^{d}\bigr) \ni F \mapsto \spanop(F) \in
    \eff(\mathbb{R}^{d}) \) is Borel.
  \item \label{item:basis-for-the-connected-component} For any \( F \in \sgr(\mathbb{R}^{d}) \),
    connected component of the origin is a vector space.  One may choose bases for these spaces in
    a Borel way: there are Borel maps \( \alpha_{i} : \sgr(\mathbb{R}^{d}) \to \mathbb{R}^{d} \), \(
    i \in \mathbb{N} \), such that for every \( F \in \sgr(\mathbb{R}^{d}) \) the set
    \[ \bigl\{ \alpha_{i}(F) : 1 \le i \le p\bigr\} \] is a basis for the connected
    component of zero in \( F \), where \( p \) is such that \( F \in \sgr_{p,q}(\mathbb{R}^{d}) \).
  \item\label{item:dimension-is-Borel} \( \sgr_{p,q}\bigl(\mathbb{R}^{d}\bigr) \) is a Borel subset of
    \( \sgr\bigl( \mathbb{R}^{d} \bigr) \) for any \( p \) and \( q \). 
  \item\label{item:basis-for-the-discrete-part} There is a Borel choice of ``basis'' for the
    discrete part of
    \( \sgr_{p,q}(\mathbb{R}^{d}) \): there are Borel maps
    \[ \beta_{i} : \sgr_{p,q}(\mathbb{R}^{d}) \to \mathbb{R}^{d}, \quad  1 \le i \le q,\]
    such that for all \( F \in \sgr_{p,q}(\mathbb{R}^{d}) \) the function
    \[ \mathbb{R}^{p} \times \mathbb{Z}^{q} \ni (a_{1}, \ldots, a_{p}, n_{1}, \ldots, n_{q}) \to
      \sum_{i=1}^{p}a_{i}\alpha_{i}(F)  + \sum_{j=1}^{q}n_{j}\beta_{j}(F) \in F \]
    is an isomorphism, where \( \alpha_{i} \) are as in
    \eqref{item:basis-for-the-connected-component}.
  \end{enumerate}
\end{lemma}

\begin{proof}
  Let \( f_{n} : \eff\bigl( \mathbb{R}^{d} \bigr) \to \mathbb{R}^{d} \) be a sequence of Borel
  selectors from the Kuratowski--Ryll-Nardzewski Theorem.

  \eqref{item:span-is-borel} For an open subset \( U \subseteq \mathbb{R}^{d} \) the set
  \( \{ F \in \eff(\mathbb{R}^{d}) : \spanop(F) \cap U \ne \es \} \) is equal to
  \[
    \begin{aligned}
      \bigl\{ F \in \eff(\mathbb{R}^{d}) :\ &\exists k_{1}, \ldots, k_{d}\ \exists a_{1}, \ldots,
      a_{d} \in \mathbb{Q} \textrm{ such that }\\ &a_{1}f_{k_{1}}(F) + \cdots + a_{d}f_{k_{d}}(F) \in
      U\bigr\}.
    \end{aligned}
\]

  \eqref{item:basis-for-the-connected-component} A basis \( \alpha_{i}(F) \) 
  can be defined by setting \( \alpha_{1}(F) = f_{m}(F) \) for the minimal \( m \in \mathbb{N} \)
  such that \( f_{m}(F) \ne \vec{0} \) and
  \[ \forall n \ge 1\ \forall \epsilon > 0\ \exists k \quad \bigl| f_{m}(F)/n - f_{k}(F) \bigr| <
    \epsilon, \]
  where we default \( \alpha_{1}(F) \) to \( \vec{0} \) if no such \( m \in \mathbb{N} \) exists.

  Continuing inductively, one sets \( \alpha_{n+1}(F) = f_{m}(F) \) for the minimal \( m \in
  \mathbb{N} \) such that \( f_{m}(F) \) is not in the span of \( \alpha_{i}(F) \), \( i \le n \),
  and \( f_{m}(F) \) is in the ``continuous part'' of \( F \):
  \[
    \begin{aligned}
      \alpha_{n+1}(F) = f_{m}(F)\ &\textrm{for the unique } m \in \mathbb{N} \textrm{ such that }\\
      &\forall N \ge 1\ \forall \epsilon > 0\ \exists k \quad \bigl| f_{m}(F)/N - f_{k}(F) \bigr| <
    \epsilon \textrm{ and}\\
      &\forall k < m\ \forall \epsilon > 0\ \exists a_{1}, \ldots, a_{n} \in \mathbb{Q}\\ &\qquad \bigl|
      a_{1}\alpha_{1}(F) + \cdots + a_{n} \alpha_{n}(F) - f_{k}(F)  \bigr| < \epsilon
      \textrm{ and}\\
      &\exists \epsilon > 0\ \forall a_{1}, \ldots, a_{n} \in \mathbb{Q}\\ &\qquad \bigl|
      a_{1}\alpha_{1}(F) + \cdots + a_{n} \alpha_{n}(F) - f_{m}(F)  \bigr| > \epsilon;
    \end{aligned}
  \]
  with the agreement that \( \alpha_{n+1}(F) = \vec{0} \) if no such \( m \) exists.

  \eqref{item:dimension-is-Borel} In view of
  \eqref{item:basis-for-the-connected-component}, the function
  \( \dim_{0} : \sgr(\mathbb{R}^{d}) \to \mathbb{N} \) that measures dimension of the connected
  component of the origin is Borel.  Thus by \eqref{item:span-is-borel} so is
  \[ \sgr_{p,q}(\mathbb{R}^{d}) = \Bigl\{ F \in \sgr(\mathbb{R}^{d}) : \dim_{0}(F) = p,\
    \dim_{0}\bigl(\spanop(F)\bigr) = p + q \Bigr\}. \]

  \eqref{item:basis-for-the-discrete-part} Let \( \alpha_{i}(F) \), \( 1 \le i \le p \), be a basis for
  the connected component of the origin in \( F \) provided by item
  \eqref{item:basis-for-the-connected-component}.  Set
  \begin{displaymath}
    \begin{aligned}
      W(F) &= \spanop\{\alpha_{i}(F) : 1 \le i \le p \}, \\
      z_{n}(F) &= f_{n}(F) - \proj_{W(F)} f_{n}(F).
    \end{aligned}
  \end{displaymath}
  Elements \( \{z_{n}(F)\} \) form a copy of \( \mathbb{Z}^{q} \) enumerated with repetitions.
  Indeed, \( \{f_{n}(F)\} \) intersects every coset \( F/W(F) \), and \( \{z_{n}(F)\} \) picks
  a unique point from each coset characterized by having a trivial projection onto \( W(F) \).  It
  therefore remains to pick a basis within \( \{z_{n}(F)\} \).

  To this end we set 
  \[ \bigl(\beta_{1}(F), \ldots, \beta_{q}(F)\bigr) = \bigl(z_{k_{1}}(F), \ldots,
    z_{k_{q}}(F)\bigr), \]
  where \( (k_{1}, \ldots, k_{q}) \) is the lexicographically least tuple such that
  \begin{itemize}
  \item \( \{z_{k_{j}}(F)\}_{j=1}^{q} \) are linearly independent;
  \item any \( z_{m}(F) \) that lies in the box
    \[ \Bigl\{ a_{1}z_{k_{1}}(F) + \cdots + a_{q}z_{k_{q}}(F) : 0 \le a_{j} \le 1,\ 1 \le j \le
      q\Bigr\}\]
    is equal to \( a_{1}z_{k_{1}}(F) + \cdots + a_{q}z_{k_{q}}(F) \) for some choice of \( a_{i} \in
    \{0,1\} \) (i.e., it is one of the vertices of the parallelepiped)
  \end{itemize}
  These conditions are easily seen to be Borel.  For instance, the last one can be written as
  \[
    \begin{aligned}
      \forall m\ &\Bigl(\exists \epsilon > 0\ \forall a_{1}, \ldots, a_{q} \in \mathbb{Q} \cap [0,1]\\
      &\qquad \bigl|a_{1}z_{k_{1}}(F) + \cdots + a_{q}z_{k_{q}}(F) - z_{m}(F)\bigr| > \epsilon\Bigr)
      \textrm{ or }\\
      & \Bigl(\exists a_{1}, \ldots, a_{q} \in \{0,1\} \quad z_{m}(F) = a_{1}z_{k_{1}}(F) + \cdots
      a_{q}z_{k_{q}}(F)\Bigr).\\  
    \end{aligned}
  \]
\end{proof}

\begin{corollary}
  \label{cor:basis-selectors}
  Let \( 0 \le p, q \le d \), \( p+q \le d\), be given.  There are Borel maps \( \alpha_{i} :
  \sgr_{p,q}(\mathbb{R}^{d}) \to \mathbb{R}^{d} \), \( 1 \le i \le p \), \( \beta_{j} :
  \sgr(\mathbb{R}^{d}) \to \mathbb{R}^{d} \), \( 1 \le j \le q \), \( \gamma_{k} :
  \sgr_{p,q}(\mathbb{R}^{d}) \to \mathbb{R}^{d} \), \( 1 \le k \le d - p - q \) such that for all \(
  F \in \sgr_{p,q}(F) \):
  \begin{enumerate}
  \item the map 
    \[ \mathbb{R}^{p} \times \mathbb{Z}^{q} \ni (a_{1}, \ldots, a_{p}, n_{1}, \ldots, n_{q}) \to
      \sum_{i=1}^{p}a_{i}\alpha_{i}(F) + \sum_{j=1}^{q}n_{j}\beta_{j}(F) \in F \] is an isomorphism;
    \item \( \bigl\{\alpha_{i}(F), \beta_{j}(F), \gamma_{k}(F)\bigr\} \) forms a basis for \( \mathbb{R}^{d} \).
  \end{enumerate}
\end{corollary}

\begin{proof}
  The first item has been proved in Lemma \ref{lem:effros-basis-selectors} above.  The second one
  is immediate by completing the linearly independent set \( \bigl\{\alpha_{i}(F), \beta_{j}(F)\bigr\}
  \) to a basis using, for example, Gramm-Schmidt orthogonalization relative to the standard basis.
\end{proof}

Let \( \mathbb{R}^{d} \acts X \) be a Borel flow on a standard Borel space \( X \), which we no
longer assume to be free.  One may consider the map \( \stab : X \to \sgr\bigl( \mathbb{R}^{d}
\bigr) \) that associates to a point \( x \in X \) its stabilizer.  This map
is known to be Borel.  Corollary \ref{cor:basis-selectors} therefore allows one to partition 
\[ X = \mkern-10mu\bigsqcup_{p,q \atop p+q \le d}\mkern-10mu X_{p,q} \]
into finitely many invariant Borel pieces, \( X_{p,q} = \{x \in X : \stab(x) \in \sgr_{p,q}(\mathbb{R}^{d})\}
\).  Moreover, and on each set \( X_{p,q} \), an orbit \( x + \mathbb{R}^{d} \) can be identified with the
quotient \( \mathbb{R}^{d}/ \stab(x) \), which is isomorphic to
\( \mathbb{R}^{r} \times \mathbb{T}^{q} \), \( r = d - p - q \).  In view of Corollary \ref{cor:basis-selectors}
we have a \emph{free} action of \( \mathbb{R}^{r} \times \mathbb{T}^{q} \) on \( X_{p,q} \), which
is defined for all \( x \in X_{p,q} \) by
\begin{displaymath}
  \begin{aligned}
    x + (s_{1}, \ldots, s_{r}, t_{1}, \ldots, t_{1}) = x + &s_{1}\gamma_{1}\bigl(\stab(x)\bigr) +
    \cdots + s_{r}\gamma_{r}\bigl(\stab(x)\bigr) + \\
    &t_{1} \beta_{1}\bigl(\stab(x)\bigr) + \cdots +
    t_{q}\beta_{q}\bigl(\stab(x)\bigr).
  \end{aligned}
\end{displaymath}
The action \( \mathbb{R}^{r} \times \mathbb{T}^{q} \acts X_{p,q} \) has the same orbits as the
action of \( \mathbb{R}^{d} \acts X_{p,q} \).  One may therefore transfer the topology (and the
smooth structure) from \( \mathbb{R}^{r} \times \mathbb{T}^{q} \) to any orbit of
\( x \in X_{p,q} \), and define a time-change equivalence between (not necessarily free) flows
\( \mathbb{R}^{d} \acts X \) and \( \mathbb{R}^{d} \acts Y \) as an orbit equivalence
\( \phi : X \to Y \) that is a homeomorphism\footnote{The topology is independent of the choice of
  bases \( \alpha_{i} \), \( \beta_{j} \), and \( \gamma_{k} \) in Corollary
  \ref{cor:basis-selectors}, but an orientation does depend on it.} on each
orbit in the sense above.

\begin{theorem}
  \label{thm:RxT-flow-tce-to-product}
  Let \( \mathbb{R}^{p} \times \mathbb{T}^{q} \acts X \) be a free Borel flow.  There exists a
  cocompressible invariant subset \( Z \subseteq X \) such that the flow restricted onto \( Z \) is
  isomorphic to a product flow.
\end{theorem}

\begin{proof}
  Let \( \mathcal{D} \subseteq X \) be a \( [-1,1]^{p}\times \mathbb{T}^{q} \)-lacunary cross
  section.  By an analog Lemma \ref{lem:uniform-rokhlin} (with only notational modifications to the
  proof), one may discard a compressible set (for convenience, we denote the remaining part by the
  same letter \( X \)) and find a sequence of cross sections \( \mathcal{C}_{n} \), and rationals
  \( l_{n} \) such that for sets \( \tilde{R}_{n} = [-l_{n},l_{n}]^{p} \),
  \( R_{n} = \tilde{R}_{n} \times \mathbb{T}^{q} \) one has
  \begin{enumerate}
  \item \( \mathcal{C}_{n} \) is \( R_{n} \)-lacunary;
  \item \( X = \bigcup_{n} \bigl(\mathcal{C}_{n} + R_{n}\bigr) \);
  \item \( \mathcal{C}_{n} + R_{n} \subseteq \mathcal{C}_{n+1} + R_{n+1}^{\la 1} \), where
    \[ R_{n+1}^{\la 1} = [-l_{n+1}+1,l_{n+1}-1]^{p}\times \mathbb{T}^{q}. \]
  \end{enumerate}

  Let \( (\epsilon_{n}) \) be a sequence of positive reals such that
  \( \sum_{n} \epsilon_{n} < 1/2 \).  We recursively construct sets \( S_{n} \subseteq X \).
  Consider a single region \( c + R_{0} \) and let the intersection
  \( (c + R_{0}) \cap \mathcal{D} \) consist of points \( \{z_{1}, \ldots, z_{m}\} \).

  Taking \( c \) to be the origin, one gets a coordinate system in \( c + R_{0} \).  Let
  \[ y_{i} = \proj_{\mathbb{R}^{p}}\rho(c, z_{i}) \textrm{ and } x_{i} =
    \proj_{\mathbb{T}^{q}}\rho(c,z_{i}), \]
  where \( \mathbb{T}^{q} = [-1,1)^{q} \).  Shifting \( c \) by a vector of norm at most
  \( \epsilon_{0} \), we may assume without loss of generality that \( x_{i} \in (-1,1)^{q} \).  Pick
  a \( C^{\infty} \)-function \( f_{0} : \tilde{R}_{0} \to (-1,1)^{q} \) such that
  \begin{enumerate}[a)]
  \item \( f_{0}(y_{i}) = x_{i} \);
  \item there is \( \delta > 0 \) such that \( f_{0} \) is constant on a \( \delta \)-collar of \(
    R_{0} \).  
  \end{enumerate}
  The same construction is performed over all regions \( c + R_{0} \), \( c \in \mathcal{C}_{0} \).
  We define \( S_{0} \) to consist of points
  \[ S_{0} = \{ c + \bigl(a, f_{0}(a)\bigr) : a \in R_{0}, c \in \mathcal{C}\}. \]
  In words, \( S_{0} \) is a surface within each of \( c + R_{0} \) that passes through points \(
  z_{i} \), it is given by a graph of a function which is constant near the boundary of its domain.

  To construct the set \( S_{1} \), consider a single \( c + R_{1} \) region,
  \( c \in \mathcal{C}_{1} \).  It contains a number of \( R_{0} \) regions, each containing a
  surface as prescribed by \( S_{0} \) (see Figure \ref{fig:construction-of-S1}).  Let
  \( T_{1}, \ldots, T_{m} \) be these surfaces.  If \( d_{1}, \ldots, d_{m}\in \mathcal{C}_{0} \)
  are such that \( d_{i} + R_{0} \subseteq c + R_{1} \), then for each \( i \), \( T_{i} \) is a
  graph of a smooth function
  \( f_{0,i} : \proj_{\mathbb{R}^{p}}\bigl(\rho(c,d_{i}) + R_{0} \bigr) \to \mathbb{T}^{q} \).

  \begin{figure}[htb]
    \centering
    \includegraphics[width=13cm]{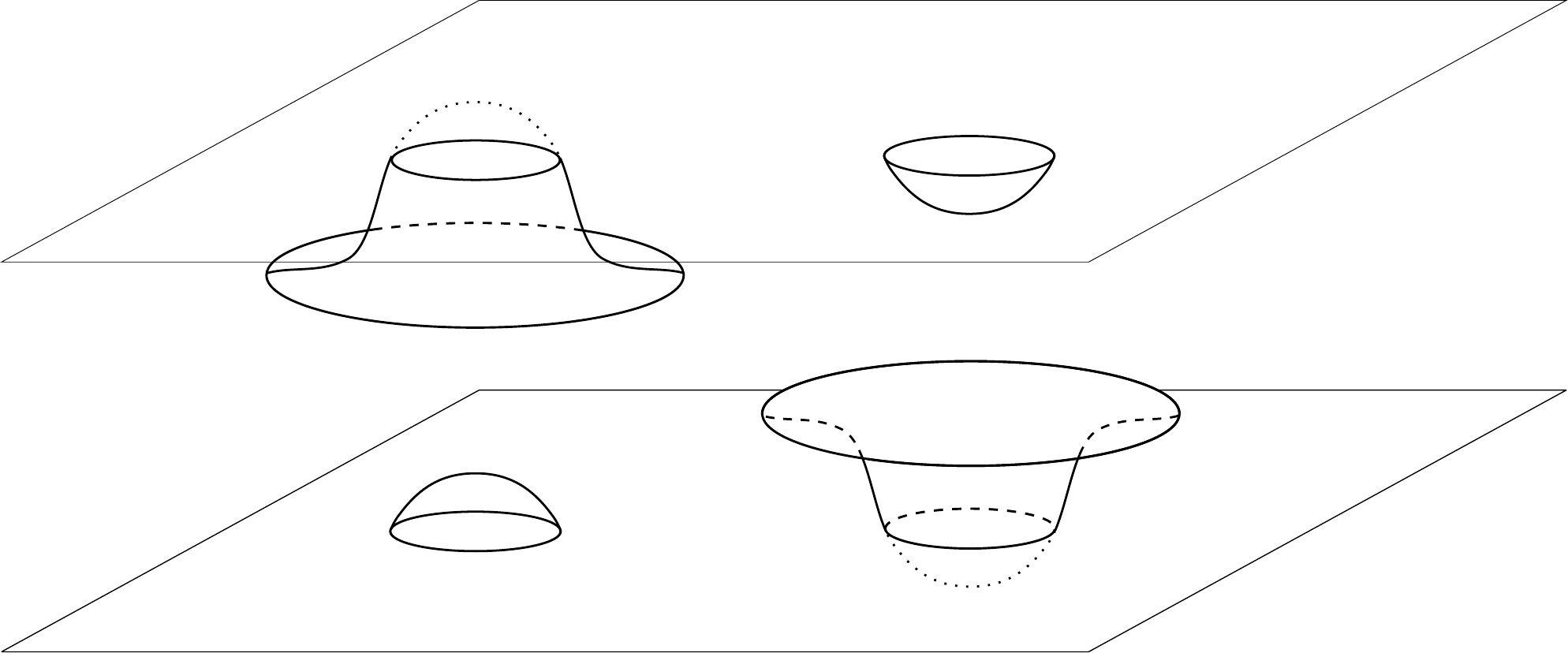}
    \caption{Construction of \( S_{1} \) for a \( \mathbb{R}^{2} \times \mathbb{T} \)-flow.}
    \label{fig:construction-of-S1}
  \end{figure}
  
  Recall that \( f_{0,i} \) are constant on a collar of \( \tilde{R}_{0} \), one may therefore shift
  \( c \) by at most \( \epsilon_{1} \) and ensure that \( f_{0,i}(x) \in (-1,1)^{q} \) for \( x \)
  in the collar of \( d_{i} + \tilde{R}_{0} \).  We therefore have \( C^{\infty} \)-functions
  \( f_{0,i} \) from disks inside \( \tilde{R_{1}} \) into \( (-1,1)^{q} \) (these functions are
  translations of the function \( f_{0} \) constructed above).  One may now extend all surfaces
  \( T_{1}, \ldots, T_{m} \) to a single smooth surface that projects injectively onto
  \( c + R_{1} \), i.e., is a graph of a smooth function \( f_{1} \).  The construction continues in
  the similar fashion.

  Set \( S \) to be the ``limit'' of \( S_{n} \) (in the same sense as in the limit of spirals of
  cross sections in Section \ref{sec:rational-grids}).  The resulting set \( S \) intersects every
  orbit of the flow in a smooth surface which is a graph of a function, i.e., \( S \) is a
  transversal for the action of \( \mathbb{T}^{q} \): if \( x + \vec{r} = y \) for some
  \( \vec{r} \in \mathbb{T}^{q} \), \( x ,y \in S \), then \( \vec{r} = \vec{0} \).  Let
  \( \zeta : X \to S \) be the selector map, such that
  \( x + \mathbb{T}^{q} = \zeta(x) + \mathbb{T}^{q} \) for all \( x \in X \).

  We now define a flow \( \flow_{0} : \mathbb{R}^{p} \acts S \) by setting
  \( \flow_{0}(x, \vec{r}) = \zeta(x + \vec{r}) \).  It is easy to see that \( \flow_{0} \) is free.
  Finally, let \( \flow' \) be the \( \mathbb{R}^{p} \times \mathbb{T}^{q} \)-flow on
  \( S \times \mathbb{T}^{q} \) given by the product of \( \flow_{0} \) and the translation on
  \( \mathbb{T}^{q} \).  The original flow \( \mathbb{R}^{p} \times \mathbb{T}^{q} \acts X  \) and
  \( \flow' \) are time change equivalent as witnessed by the natural identification between \( X \)
  and \( S \times \mathbb{T}^{q} \).
\end{proof}

\begin{corollary}
  \label{cor:time-change-of-periodic}
  Let \( \flow_{i} : \mathbb{R}^{d} \acts X_{i} \), \( i = 1, 2 \), be Borel flows.  For \( p, q \in \mathbb{N} \)
  such that \( p + q \le d \) let
  \[ Y_{i}(p,q) = \{x \in X_{i} : \stab(x) \in \sgr_{p,q}(\mathbb{R}^{d})\}. \]
  Suppose that for each pair \( (p, q) \) one of the following is true:
  \begin{itemize}
  \item The restriction of \( \flow_{i} \) onto \( Y_{i}(p,q) \) is smooth, and also flows \(
    \flow_{1}|_{Y_{1}(p,q)} \) and \( \flow_{2}|_{Y_{2}(p,q)} \) have the same number of orbits;
  \item Both \( \flow_{i}|_{Y_{i}(p,q)} \) are non smooth.
  \end{itemize}
  In this case flows \( \flow_{i} \) are time change equivalent up to a compressible perturbation.
\end{corollary}

\begin{proof}
  The proof is immediate from Theorem \ref{thm:RxT-flow-tce-to-product}, Theorem
  \ref{thm:time-change-equivalence-by-compressible} and Corollary \ref{cor:basis-selectors}.
\end{proof}

\section{Orbit equivalences of \( \mathbb{T} \)-flows}
\label{sec:loe-vs-tce}

In this last section we show how Lebesgue orbit equivalence, defined as orbit equivalence that
preserves the Lebesgue measure between orbits, exhibits a completely different behavior than time
change equivalence.

Recall the following notion from \cite{2015arXiv150400958S}: two \emph{free} Borel flows
\( \mathbb{R} \acts X \) and \( \mathbb{R} \acts Y \) are said to be \emph{Lebesgue orbit
  equivalent} if there exists an orbit equivalence bijection \( \phi : X \to Y \) which preserves
the Lebesgue measure on every orbit.  Freeness is needed to transfer the Lebesgue measure from
\( \mathbb{R} \) to orbits of the flow.

In general, any Borel flow \( \mathbb{R} \acts X \) can be decomposed into a periodic and aperiodic
parts, i.e., there is a Borel partition \( X = X_{1} \sqcup X_{2} \) into invariant pieces such that
\( \mathbb{R} \acts X_{2} \) is free, while \( \mathbb{R} \acts X_{1} \) is periodic, i.e., for any
\( x \in X_{1} \) there is some \( \lambda \in \mathbb{R} \setminus \{0\} \) such that
\( x + \lambda = \lambda \) (this is a simple instance of item \eqref{item:dimension-is-Borel} in
Lemma \ref{lem:effros-basis-selectors}) We may therefore define a map
\( \per : X_{1} \to \mathbb{R}^{\ge 0} \) by
\[ \per(x) = \inf\{ \lambda \in \mathbb{R}^{>0} : x + \lambda = x\}. \]
The \emph{period} map \( \per \) is easily seen to be Borel.  The set of fixed points by the flow is
characterized by the equation \( \per(x) = 0 \).

For convenience, let us say that a flow \( \mathbb{R} \acts X \) is \emph{purely periodic} if any
\( x \in X \) is periodic and there are no fixed
points for the flow.  An orbit of a point \( x \) can therefore be naturally identified with an
interval \( \bigl[0, \per(x)\bigr) \) and endowed with a Lebesgue measure on this interval
(\emph{not} normalized).  We
obtain a Borel assignment of measures \( x \mapsto \mu_{x} \), which is invariant under the action
of the flow.  It is natural to extend the concept of Lebesgue orbit equivalence to purely periodic
flows by declaring two of them \( \mathbb{R} \acts X \), \( \mathbb{R} \acts Y \) to be
\emph{Lebesgue orbit equivalent} whenever there is a bijection \( \phi : X \to Y \) which preserves
the orbit equivalence relation and satisfies \( \phi_{*}\mu_{x} = \mu_{\phi(x)} \) for all
\( x \in X \), i.e., it preserves the Lebesgue measure within every orbit.

In the case of discrete actions \( \mathbb{Z} \acts X \), the above definition corresponds to the
requirement of preserving the counting measure within every periodic orbit.  This is automatically
satisfied by any bijection that preserves orbits.  Since there are only countably many possible
sizes of orbits, two periodic hyperfinite equivalence relations, \( \cber \) on \( X \) and \( \fber \) on
\( Y \), are isomorphic if and only for any \( n \in \mathbb{N} \) sets 
\[ \{x \in X : \bigl| [x]_{\cber} \bigr| = n \} \quad \textrm{and} \quad \{ y \in Y : \bigl| [y]_{\fber}
  \bigr| = n \}\]
have the same size.

The analog of the condition above is also obviously necessary for purely periodic flows to be
Lebesgue orbit equivalent: for any \( \lambda \in \mathbb{R}^{> 0} \) cardinalities of orbits of
period \( \lambda \) have to be the same in both flows.  The purpose of this section is to show that
contrary to the discrete case, for purely periodic Borel flows this condition is no longer
sufficient.

Let \( \mathbb{R} \acts X \) be a purely periodic Borel flow, let \( \cber \) denote its orbit
equivalence relation, and let \( Z \subseteq X \) be a Borel transversal for \( \cber \).  The pair
\( (Z, \per) \), where \( \per \) is the restriction of the period function onto \( Z \), completely
characterizes the flow.  Indeed, the flow can be recovered as a flow under the function
\( \per : Z \to \mathbb{R}^{>0} \) with the trivial base automorphism (see \cite[Chapter 7]{MR1725389}).  The converse is also true:
any pair \( (Z, \per) \), where \( Z \) is a standard Borel space and
\( \per : Z \to \mathbb{R}^{> 0}\) is a Borel function, gives rise to a purely periodic Borel
automorphism.  Since any Lebesgue orbit equivalence between purely periodic flows has to preserve
the period map, the problem of classifying purely periodic flows up to Lebesgue orbit equivalence
can therefore be reformulated as a problem of classifying all pairs \( (Z, f) \), where \( Z \) is a
standard Borel space and \( f : Z \to \mathbb{R}^{> 0} \) is a Borel map, up to isomorphism, i.e.,
up to existence of a Borel bijection \( \phi : Z_{1} \to Z_{2} \) such that
\( \phi\bigl(f_{1}(x)\bigr) = f_{2} \bigl( \phi(x)\bigr) \) for all \( x \in Z_{1} \).  

Our necessary condition for Lebesgue orbit equivalence transforms into the following: if \( (Z_{1},
f_{1}) \) and \( (Z_{2}, f_{2}) \) are isomorphic, then \( \bigl| f_{1}^{-1}(\lambda) \bigr|  =
\bigl| f_{2}^{-1}(\lambda) \bigr| \) for all \( \lambda \in \mathbb{R}^{>0} \).  

\begin{proposition}
  \label{prop:not-sufficient}
  There are two pairs \( (Z_{1}, f_{1}) \) and \( (Z_{2}, f_{2}) \), where \( Z_{i} \) are standard
  Borel spaces and \( f_{i} : Z_{i} \to \mathbb{R}^{>0} \) are Borel maps, such that \( \bigl|
  f_{1}^{-1}(\lambda) \bigr| = \bigl| f_{2}^{-1}(\lambda) \bigr| \) for all \( \lambda \in
  \mathbb{R}^{>0} \) and yet \( (Z_{1}, f_{1}) \) and \( (Z_{2}, f_{2}) \) are not isomorphic.
\end{proposition}

\begin{proof}
  Let \( Z_{1} \subseteq [1,2] \times \mathbb{R} \) be a Borel set which admits no Borel
  uniformization (see \cite[Section 18]{MR1321597}) and satisfies for all \( x \in [1,2] \):
  \[ \bigl| \{ y \in \mathbb{R} : (x,y) \in Z_{1} \} \bigr| = \mathfrak{c}. \]
  Existence of such sets is well-known (see, for example, Exercise 18.9 and Exercise 18.17 in
  \cite{MR1321597}).  Let \( Z_{2} \subseteq [1,2] \times \mathbb{R} \) be any Borel set which does
  admit a Borel uniformization and satisfies for all \( x \in [1,2] \):
  \[ \bigl| \{ y \in \mathbb{R} : (x,y) \in Z_{2} \} \bigr| = \mathfrak{c}. \]
  For instance, one may take \( Z_{2} = [1,2] \times \mathbb{R} \).

  Let \( f_{i} : Z_{i} \to [1,2] \) be  projections onto the first coordinate.  Pairs
  \( (Z_{1}, f_{1}) \) and \( (Z_{2}, f_{2}) \) are not isomorphic, because by construction the
  relation on \( Z_{1} \) given by \( x \sim y \) whenever \( f_{1}(x) = f_{1}(y) \) does not admit a
  Borel transversal, while the analogous relation on \( Z_{2} \) admits one.
\end{proof}

In contrast, time change equivalence relation on purely periodic flows is, of course, trivial.

\begin{proposition}
  \label{prop:time-change-for-purely-periodic}
  Let \( \mathbb{R} \acts X_{1} \) and \( \mathbb{R} \acts X_{2} \) be purely periodic flows.  If
  cardinalities of orbit spaces of these flows are the same, then they are time change equivalent.
\end{proposition}

\begin{proof}
  Let \( D_{i} \subseteq X_{i} \) be transversals for the orbit equivalence relations.  By
  assumption \( |D_{1}| = |D_{2}| \), so let \( \phi : D_{1} \to D_{2} \) be any Borel bijection.
  Let \( \xi_{i} : X_{i} \to D_{i} \) be Borel selectors, and let
  \( \rho_{i} : X_{i} \to \mathbb{R} \) be such that \( \xi_{i}(x) + \rho_{i}(x) = x \),
  \( \rho_{i}(x) \in \bigl[0,\per(x)\bigr) \) for all \( x \in X_{i} \).  Extend \( \phi \) to a
  time change equivalence by setting
  \[ \phi(x) = \phi\bigl(\xi_{1}(x)\bigr) + \rho_{1}(x) \times \frac{\per(\xi_{1}(x))}{\per(x)}. \qedhere\]
\end{proof}

\bibliographystyle{amsalpha}
\bibliography{refs.bib}

\end{document}